\documentclass[11pt]{article}
\usepackage{amsmath, amssymb, amsthm,enumerate}

\date{}
\title{\vspace{-0.9cm}Almost-Fisher families}
\author{
Shagnik Das \thanks{Department of Mathematics, ETH, 8092 Zurich, Switzerland. Email: shagnik@ucla.edu.}
\and
Benny Sudakov \thanks{Department of Mathematics, ETH, 8092 Zurich, Switzerland. Email: benjamin.sudakov@math.ethz.ch. Research supported in part by SNSF grant 200021-149111 and by a USA-Israel BSF grant.}
\and
Pedro Vieira \thanks{Department of Mathematics, ETH, 8092 Zurich, Switzerland.  Email: pedro.vieira@math.ethz.ch.}
}

\allowdisplaybreaks

\theoremstyle{plain}
\newtheorem{thm}{Theorem}[section]
\newtheorem{prop}[thm]{Proposition}
\newtheorem{lemma}[thm]{Lemma}

\newtheorem{cor}[thm]{Corollary}

\theoremstyle{definition}

\newcommand{\floor}[1]{\left\lfloor #1 \right\rfloor}
\newcommand{\ceil}[1]{\left\lceil #1 \right\rceil}
\newcommand{\card}[2][]{\left| #2 \right|_{#1}}
\newcommand{\comp}[1]{\overline{#1}}

\newcommand{\cF}{\mathcal{F}}
\newcommand{\cG}{\mathcal{G}}
\newcommand{\cH}{\mathcal{H}}
\newcommand{\cP}{\mathcal{P}}
\newcommand{\cQ}{\mathcal{Q}}
\newcommand{\cR}{\mathcal{R}}
\newcommand{\cS}{\mathcal{S}}
\newcommand{\rank}{\mathrm{rank}}
\newcommand{\RR}{\mathbb{R}}

\begin{document}
\maketitle

\begin{abstract}
A classic theorem in combinatorial design theory is Fisher's inequality, which states that a family $\cF$ of subsets of $[n]$ with all pairwise intersections of size $\lambda$ can have at most $n$ non-empty sets.  One may weaken the condition by requiring that for every set in $\cF$, all but at most $k$ of its pairwise intersections have size $\lambda$. We call such families $k$-almost $\lambda$-Fisher. Vu was the first to study the maximum size of such families, proving that for $k=1$ the largest family has $2n-2$ sets, and characterising when equality is attained.  We substantially refine his result, showing how the size of the maximum family depends on $\lambda$. In particular we prove that for small $\lambda$ one essentially recovers Fisher's bound.  We also solve the next open case of $k=2$ and obtain the first non-trivial upper bound for general $k$.
\end{abstract}

\section{Introduction} \label{sec:intro}

A \emph{$\lambda$-Fisher family} is a family $\cF$ of sets with $\card{F_1 \cap F_2} = \lambda$ for all distinct $F_1, F_2 \in \cF$.  A fundamental result in combinatorial design theory, Fisher's inequality shows that a $\lambda$-Fisher family of subsets of $[n]$ can contain at most $n$ non-empty sets.  This simple restriction thus severely constricts the size of the family.  One might hope to find larger families by weakening the conditions somewhat, allowing a limited number of `bad' intersections.  

To this end, we define a \emph{$k$-almost $\lambda$-Fisher family} as a family $\cF$ of sets such that for every $F \in \cF$, there are at most $k$ other sets $F' \in \cF$ with $\card{F \cap F'} \neq \lambda$.  We may then extend Fisher's inequality by determining how large a $k$-almost $\lambda$-Fisher family over $[n]$ can be.  We denote this maximum by $f(n,k,\lambda)$.  This problem was first introduced by Vu in 1999, who solved the case $k = 1$.  In this paper we sharpen Vu's result, determining how the size of the maximum family depends on $\lambda$ when $k=1$, and solve the next open case of $k=2$. We also obtain the first non-trivial upper bound for 
general $k$ and provide a tight estimate for $\lambda=0$.

We now discuss the background of this problem in greater detail before presenting our new results.

\subsection{Restricted intersections}

Extremal set theory is a rapidly developing area of combinatorics and has enjoyed tremendous growth in the past few decades.  No doubt this is fuelled by its deep connections to other areas; extremal set theory both employs methods from and enjoys applications to diverse fields such as algebra, geometry and coding theory.

Many problems in extremal set theory are concerned with the pairwise intersections between sets in a family.  For instance, much research concerns intersecting families, where empty pairwise intersections are forbidden.  An intersecting family cannot contain a complementary pair of sets, and so can have size at most $2^{n-1}$, a bound attained by many constructions.  The celebrated Erd\H{o}s--Ko--Rado theorem \cite{ekr61}, one of the cornerstones of extremal set theory, provides the corresponding extremal result for $k$-uniform intersecting families.

Rather than forbidding empty intersections, we might seek to forbid other intersection sizes instead.  Note that we would expect two uniformly random subsets of $[n]$ to intersect in $n/4$ elements.  Erd\H{o}s asked whether a family without a pairwise intersection of size exactly $n/4$ must have exponentially fewer than $2^n$ sets.  This conjecture was resolved by Frankl and R\"odl \cite{fr87}, who obtained a stronger result in the more general setting of codes (see \cite{kl14} for a recent improvement).

As opposed to forbidding intersection sizes, one might instead require that all pairwise intersections be of the same size, and unsurprisingly this proves to be a much more restrictive condition.  We call a family of sets a \emph{$\lambda$-Fisher family} if any two distinct sets intersect in $\lambda$ elements.  The foundational result in this direction is Fisher's inequality \cite{f40}, bounding the size of a $\lambda$-Fisher family.  Fisher's original result dealt with more restrictive designs, and was extended to uniform $\lambda$-Fisher families by Bose \cite{b49}.  The following non-uniform version was proven by Majumdar \cite{m53} and rediscovered by Isbell \cite{i59}.

\begin{thm} \label{thm:fisher}
A $\lambda$-Fisher family over $[n]$ can have at most $n$ non-empty sets.
\end{thm}

Note that the non-empty condition is necessary, as when $\lambda = 0$, one may take the empty set in addition to the $n$ singletons.  However, when $\lambda > 0$, a $\lambda$-Fisher family cannot contain the empty set.  The classification of extremal constructions remains one of the most important problems in combinatorial design theory.  In the case $\lambda = 1$, this reduces to the famous de Bruijn--Erd\H{o}s theorem \cite{de46}, for which the extremal constructions are known precisely.

Theorem \ref{thm:fisher} has inspired a great deal of research, having been extended in numerous directions by several renowned mathematicians.  Ray-Chaudhuri and Wilson \cite{rw75} provided bounds on families where $s$ different intersection sizes are allowed, and Frankl and Wilson \cite{fw81} considered the problem where the sizes of sets and intersections are taken modulo a prime $p$.  These results have proven remarkably useful in the field of discrete geometry.  For a more detailed history of this branch of research, we refer the reader to Babai and Frankl's monograph on the linear algebra method \cite{bf92}.

Another extension that has attracted plenty of attention in recent years is to restrict the sizes of $k$-wise intersections instead of pairwise intersections.  This problem was first raised by S\'os \cite{s76}, and resolved in the uniform setting by F\"uredi \cite{f83}.  Vu \cite{v99} studied the problem for intersections modulo $2$, and Grolmusz and Sudakov \cite{gs02} extended this to systems modulo an arbitrary prime.  Exact results in this setting were obtained by Szab\'o and Vu \cite{sv05}.  In the non-modular setting, asymptotically sharp results were provided by F\"uredi and Sudakov \cite{fs04}.

\subsection{Almost-Fisher families}

We seek a different extension of Fisher's inequality.  As we have seen, requiring that \emph{all} pairwise intersections in a set family $\cF$ have size $\lambda$ severely restricts the size of $\cF$.  One might hope to find larger families by relaxing the condition and allowing a few `bad' intersections to appear.  This practice of weakening the conditions locally is widespread in combinatorics; for instance, Alon, Jiang, Miller and Pritikin \cite{ajmp03} generalised the concept of proper edge-colourings of graphs by allowing every colour to be used a bounded number of times at any vertex, while Gerbner, Lemons, Palmer, Patk\'os and Sz\'ecsi \cite{glpps12} and Scott and Wilmer \cite{sw13} have studied the corresponding weakened version of the Erd\H{o}s--Ko--Rado theorem, bounding the size of families with each set being disjoint from a fixed number of other sets.

The weakened version of Fisher's inequality was first introduced by Vu \cite{v99}.  Recall a family $\cF$ of sets is \emph{$k$-almost $\lambda$-Fisher} if for every set $F \in \cF$ there are at most $k$ other sets $F' \in \cF \setminus \{F\}$ for which $\card{F \cap F'} \neq \lambda$.  In particular, note that when $k = 0$ this reduces to the $\lambda$-Fisher families defined previously.  We are interested in bounding the size of a $k$-almost $\lambda$-Fisher family over $[n]$, and denote the largest possible size by
\[ f(n,k,\lambda) = \max \left\{ \card{\cF} : \cF \subset 2^{[n]} \textrm{ is a $k$-almost $\lambda$-Fisher family} \right\}. \]

Most results regarding restricted intersections are proven by linear algebraic methods, using the restrictions on the system to build a linearly independent set of vectors in an appropriate vector space.  Dimensional arguments then provide the required bound on the size of the set family.  Allowing some intersections of different sizes destroys the linear independence of these vectors.  However, by marrying the algebraic arguments with some graph theoretic considerations, we can still recover some bounds on $f(n,k,\lambda)$.

Given a $k$-almost $\lambda$-Fisher family $\cF$, we can define an auxiliary graph $G = G(\cF) = (V,E)$, where $V = \cF$ and $\{F, F'\} \in E$ if and only if $\card{F \cap F'} \neq \lambda$.  Since every set in $\cF$ can have at most $k$ pairwise intersections not equal to $\lambda$, it follows that the maximum degree of $G$ is at most $k$.  Moreover, an independent set in $G$ corresponds to a $\lambda$-Fisher subfamily of $\cF$.  Since any such family can have at most $n$ non-empty sets, and any graph with $m$ vertices and maximum degree $\Delta$ has an independent set of size at least $m / (\Delta + 1)$, we arrive at the following upper bound for $f(n,k,\lambda)$, first given in \cite{v99}.

\begin{prop} \label{prop:trivial}
For any positive integer $n$ and non-negative integers $k$ and $\lambda$, $f(n,k,\lambda) \le (k + 1)n + 1$.  Moreover, if $\lambda \neq 0$, we can improve this bound to $f(n,k,\lambda) \le (k+1)n$.
\end{prop}

Vu \cite{v99} showed that this essentially gives the correct bound when $k = 1$, and was further able to prove that the extremal constructions arise from Hadamard matrices (we describe this construction in greater detail in Section \ref{sec:k1prelim}).

\begin{thm}\label{thm:vu}
For $n \ge 3$ and for any non-negative $\lambda$, $f(n,1,\lambda) \le 2(n-1)$.  Moreover, if $n \ge 4$ and equality holds, then $\lambda = n/4$ and a Hadamard matrix of order $n$ exists.
\end{thm}

For large values of $k$, however, Vu noted that Proposition \ref{prop:trivial} appears to be far from the truth, and asked to determine the correct 
behaviour of $f(n,k,\lambda)$.

\subsection{Our results}

In this paper we continue the study of $k$-almost $\lambda$-Fisher families, approaching the problem of determining $f(n,k,\lambda)$ from a few different directions.  In doing so, we are able to substantially refine Vu's result, while also obtaining some evidence that $nk/4$ might be the correct asymptotic behaviour for large $k$.

As a warm up, we start with the case $\lambda = 0$, and for brevity call a $k$-almost $0$-Fisher family a $k$-almost disjoint family.  We obtain the following bound on $f(n,k,0)$, and show for every $k$ that this is tight for infinitely many values of $n$.

\begin{thm} \label{thm:almostdisjoint}
For any positive integers $n$ and $k$, we have
\[ f(n,k,0) \le \frac{n}{k} \floor{\frac{k^2}{4}} + n + 1. \]
\end{thm}

In particular, this shows that as $k$ grows, the largest $k$-almost disjoint family has size asymptotically $nk/4$, matching a construction of Vu (for which $\lambda = 2$).

We next turn our attention to the case $k = 1$.  Recall that Vu proved $f(n,1,\lambda) \le 2n - 2$, attainable only if $\lambda = n/4$. It is very natural to ask what happens for other values of $\lambda$, i.e., to study the dependence of the function $f(n,1,\lambda)$ on the parameter $\lambda$.
Here we essentially resolve this question, obtaining the following result (whose tightness we discuss in Section 3.1).

\begin{thm} \label{thm:kequalsone}
For integers $n \ge 1$ and $\lambda \ge 0$, we have
\[ f(n,1,\lambda) \le \max \left\{ n + 2, 8 \min \left\{ \lambda, \frac{n- \lambda}{3} \right\} + o ( \lambda ) \right\}. \]
\end{thm}

Note that $f(n, 1, \lambda)$ is only close to $2n$ when $\lambda$ is close to $n / 4$, providing stability for Theorem \ref{thm:vu}.  Moreover, if $\lambda < n/8$ or $\lambda > 5n/8$, then allowing one non-$\lambda$ intersection per set provides almost no gain compared to Fisher's inequality.

We provide further evidence of the Hadamard construction being atypically large by extending Vu's methods and showing that it is the best possible even when we allow two bad intersections per set.  Once again, we provide stability by showing that $f(n,2,\lambda)$ is much smaller than $2n$ when $\lambda$ is far from $n/4$.  We also show that when $\lambda = o(n)$, $f(n,2,\lambda) = \left(\frac32 + o(1) \right) n$, which is asymptotically the bound obtained when $\lambda = 0$ in Theorem \ref{thm:almostdisjoint}.  This suggests that perhaps the $\lambda = 0$ case exemplifies the true behaviour of the function $f(n,k,\lambda)$ for large $k$, implying that $nk/4$ is the correct bound.

\begin{thm} \label{thm:kequalstwo}
For $n$ sufficiently large and $0 \le \lambda \le n$, we have the bounds
\begin{itemize}
	\item[(i)] $f(n,2,\lambda) \le 2n - 2$.
	\item[(ii)] $f(n,2,\lambda) \le \frac13 \left( 5n + 4 \min \left\{ \lambda,\frac{n-\lambda}{3} \right\} + 7 \right)$.
	\item[(iii)] $f(n,2,\lambda) \le \left( \frac32 + o(1) \right) n$ when $\lambda = o(n)$.
\end{itemize}
\end{thm}

Finally, we are able to use our results to give the first non-trivial upper bound for general $k$.

\begin{cor} \label{cor:largek}
For $k \ge 1$, we have $f(n,k,\lambda) \le (2n - 2) \ceil{\frac{k+1}{3}}$.  Moreover, if $\lambda = o(n)$, then $f(n,k,\lambda) \le \left(\frac32 + o(1)\right) n \ceil{\frac{k+1}{3}}$.
\end{cor}

\subsection{Outline and notation}
\label{subsec:outline and notation}

Our paper is organized as follows.  In Section \ref{sec:almostdisjoint}, we prove Theorem \ref{thm:almostdisjoint}, bounding the size of $k$-almost disjoint families.  In Section \ref{sec:kequalsone}, we carefully analyse the structure of $1$-almost $\lambda$-Fisher families and prove Theorem \ref{thm:kequalsone}.  In Section \ref{sec:kequalstwo}, we extend these arguments to $2$-almost $\lambda$-Fisher families, proving Theorem \ref{thm:kequalstwo}.  In Section \ref{sec:largek} we deduce Corollary \ref{cor:largek}.  In the final section we provide some concluding remarks and open questions.  Some technical lemmas needed in Section \ref{sec:kequalstwo} are proven in Appendix \ref{app:appendix}.

Although we are studying set families $\cF = \{F_1, F_2, \hdots, F_m\}$ over the ground set $[n] = \{1, 2, \hdots, n\}$, we shall often think about them in terms of the auxiliary graph $G(\cF)$ or in terms of the characteristic vectors of the sets over $\{0,1\}^n$.  Recall that the auxiliary graph $G(\cF)=(V,E)$ has $V = \cF$, with an edge $\{F, F'\}$ if and only if $\card{F \cap F'} \neq \lambda$.  We shall use $P_t$ to denote a path on $t+1$ vertices and $C_t$ to denote a cycle of $t$ vertices.

Given a set $F$, its characteristic vector $v_F \in \{0,1\}^n$ is defined by $(v_F)_j = 1$ for all $j \in F$ and $(v_F)_j = 0$ for $j \notin F$.  If our sets are indexed, we will sometimes abbreviate notation by writing $v_i = v_{F_i}$.  Given a set family $\cF = \{F_1, F_2, \hdots, F_m\}$, the \emph{element-set incidence matrix} $A(\cF)$ is an $n \times m$ matrix whose columns are the characteristic vectors $v_i$.  Thus $A(\cF)$ is a $\{0,1\}$-matrix with $A(\cF)_{j,i} = 1$ if and only if $j \in F_i$.

We then define the \emph{intersection matrix} $M(\cF)$ to be an $m \times m$ matrix with $M(\cF)_{i,j} = \card{F_i \cap F_j} - \lambda$.  Note that $M(\cF) = A(\cF)^T A(\cF) - \lambda J_m$, where $J_m$ is the $m \times m$ matrix with all entries equal to $1$.  Many of our arguments will rely on the rank of the intersection matrix over the reals, and hence any reference to the rank should be understood as the rank of $M(\cF)$ over $\mathbb{R}$, unless otherwise stated.

When the underlying set family is understood from context, we will often suppress it in the notation, writing $G$, $A$ and $M$ instead.  For brevity, we will often combine terminology from graph theory and linear algebra when referring to our set families.  For instance, we will call a set family $\cF$ a \emph{rank-$3$ $C_5$} if $\rank (M(\cF)) = 3$ and the corresponding auxiliary graph $G$ is a cycle on five vertices.

For a set family $\cF$ and $j \in [n]$, we write $\cF(j) = \{F \in \cF : j \in F\}$ for the subfamily consisting of sets containing $j$.  We will sometimes wish to restrict our families to subsets $X \subset [n]$ of the ground set, and shall write $\card[X]{F} = \card{F \cap X}$.  Finally, all logarithms are in base two.

\section{Almost disjoint families} \label{sec:almostdisjoint}

In this section we set $\lambda = 0$ and determine the largest possible size of a $k$-almost disjoint family over $[n]$.  We shall prove the tight bound in Theorem \ref{thm:almostdisjoint}, which we recall below.

\setcounter{section}{1}
\setcounter{thm}{3}

\begin{thm}
For any positive integers $n$ and $k$, we have
\[ f(n,k,0) \le \frac{n}{k} \floor{ \frac{k^2}{4}} + n + 1. \]
\end{thm}

\setcounter{section}{2}
\setcounter{thm}{0}

We show this bound is tight whenever $k$ is even, $4 | kn$ and $n \ge k/2 + 2$, or $k$ is odd and $k | n$.  Construct a family $\cF$ by taking the edges of a graph $H$ with vertex set $[n]$, along with the empty set and all $n$ singletons in $[n]$.  If $k$ is even, we may take $H$ to be any $(k/2)$-regular graph, which exists whenever $kn/2$ is even and $k/2 \le n-1$.  If $k$ is odd, let $H$ be a disjoint union of $K_{(k-1)/2, (k+1)/2}$'s.  It is easy to verify that these families are $k$-almost disjoint with the desired number of sets.  We now prove Theorem \ref{thm:almostdisjoint}, showing that these constructions are the best possible.

\begin{proof}
We prove the theorem by induction on the number of sets of size at least $4$.  For the induction step, suppose we have a set $F \in \cF$ with $s = \card{F} \ge 4$.  Let $\cF_0 = \{ F' \in \cF : F' \cap F = \emptyset\}$ and $\cF_1 = \cF \setminus \cF_0$.  By assumption, $\card{\cF_1} \le k + 1$.  By induction, since $\cF_0$ is a $k$-almost disjoint system over $n - s$ elements with fewer sets of size at least $4$, $\card{\cF_0} \le \frac{n-s}{k} \floor{\frac{k^2}{4}} + n - s + 1$.  Thus we have
\begin{align*}
  \card{\cF} = \card{\cF_0} + \card{\cF_1} &\le \frac{n-s}{k} \floor{\frac{k^2}{4}} + n - s + 1 + k + 1 \\
  &\le \frac{n}{k} \floor{\frac{k^2}{4}} + n + 1 - s \left( \frac{k^2-1}{4k} + 1 \right) + k + 1 < \frac{n}{k} \floor{\frac{k^2}{4}} + n + 1,
\end{align*} 
since $s \ge 4$.

For the base case, we may assume that $\card{F} \le 3$ for all $F \in \cF$.  Replacing any $F \in \cF$ with a subset $F' \subset F$ not already in the family does not increase the number of intersections for any set.  Hence we may further assume $\cF$ is a down-set, with $F \in \cF$ and $F' \subset F$ implying $F' \in \cF$ as well.

Given any $3$-set $T = \{u,v,w\} \in \cF$ and any $S \subset T$, let $\cF_{T}(S) = \{ F \in \cF : F \cap T = S\}$.  Since $\cF$ is a down-set, we have $\card{\cF_{T}(S')} \ge \card{\cF_{T}(S)}$ for any $S' \subset S$, as we have the injection $\Phi: \cF_T(S) \rightarrow \cF_T(S')$ given by $\Phi(F) = F \setminus ( S \setminus S')$.  Hence we have $\card{\cF_T(\{u,v\})} \le \card{\cF_T(\{u\})}$, $\card{\cF_T(\{v,w\})} \le \card{\cF_T(\{v\})}$, and $\card{\cF_T(\{u,w\})} \le \card{\cF_T(\{w\})}$.  Summing and rearranging, we have
\[ \sum_{x \in T} \left[ \card{\cF_T(\{x\})} - \card{\cF_T(T \setminus \{x\})} \right] \ge 0, \]
and so there is some $x \in T$ with $\card{\cF_T(\{x\})} \ge \card{\cF_T(T \setminus \{x\})}$.

Form a multigraph $\cG$ on vertices $[n]$ from $\cF$ by removing the empty set and singletons, and replacing every $3$-set $T$ by the ($2$-)edge $T \setminus \{x\}$, where $x$ is as above.  Note that, since we assumed $\cF$ is a down-set, every replaced $3$-set gives rise to a repeated edge.  For a vertex $v$, let $d_v$ denote its degree (with multiplicity) in $\cG$.  We now apply the following result.

\begin{prop} \label{prop:degreesum}
Let $\cG$ be a multigraph on $n$ vertices.  If $e(\cG) > \frac{n}{k} \floor{\frac{k^2}{4}}$, then there is some edge $\{u,v\} \in E(G)$ with $d_u + d_v \ge k+1$.
\end{prop}

We have $e(\cG) \ge \card{\cF} - n - 1$, and so if $\card{\cF} > \frac{n}{k} \floor{\frac{k^2}{4}} + n + 1$, the hypotheses of Proposition \ref{prop:degreesum} are satisfied, and thus we have some edge $\{u,v\} \in \cG$ with $d_u + d_v \ge k + 1$.

By inclusion-exclusion, the number of edges in $\cG$ intersecting $\{u,v\}$ is $d_u + d_v - d_{uv}$, where $d_{uv}$ is the multiplicity of the edge $\{u,v\}$ itself.  If $d_{uv} = 1$, then the edge $\{u,v\}$ intersects at least $k$ edges in $\cG$, each of which corresponds to a set in $\cF$.  This counts the set $\{u,v\}$ itself, but does not count the two singletons $\{u\}$ and $\{v\}$ (which belong to $\cF$ as $\cF$ is a down-set), and hence the set $\{u,v\}$ intersects at least $k+1$ other sets in $\cF$, a contradiction.

Hence we may assume $d_{uv} \ge 2$, which is only possible if some $3$-set $T = \{u,v,w\}$ in $\cF$ is mapped to the edge $\{u,v\}$ in $\cG$.  In this case, we consider the number of sets in $\cF$ that intersect $T$.  Each of the edges in $\cG$ intersecting $\{u,v\}$ corresponds to a different set in $\cF$ that intersects $T$.  There are at least $d_u + d_v - d_{uv} \ge k+1 - d_{uv}$ such sets.

Note that the sets in $\cF$ which reduce to $\{u,v\}$ in $\cG$ are $T$ itself and some of the sets in $\cF_T(\{u,v\})$.  On the other hand, the sets corresponding to those in $\cF_T(\{w\})$ do not intersect $\{u,v\}$ in $\cG$, and by construction, $\card{\cF_T(\{w\})} \ge \card{\cF_T(\{u,v\})}$.  Hence we overcount $d_{uv} \le \card{\cF_T(\{u,v\})} + 1$ sets, but undercount $\card{\cF_T(\{w\})}$ sets.  Accounting for the fact that we shouldn't count the $3$-set $T$ itself, but should count the three singletons $\{u\}$, $\{v\}$, and $\{w\}$, it follows that there are at least $k+2$ other sets in $\cF$ intersecting $T$, a contradiction.  This completes the proof.

\end{proof}

It remains to prove Proposition \ref{prop:degreesum}.

\begin{proof}[Proof of Proposition \ref{prop:degreesum}]

We will in fact prove that if $\cG$ is a multigraph with $\sum_{u \sim v} (d_u + d_v) \le ke(\cG)$, where the sum is taken with multiplicity over repeated edges, then $e(\cG) \le \frac{n}{k} \floor{\frac{k^2}{4}}$.  The claim then follows by averaging over the edges.

To begin, we show that we may assume all the degrees are nearly equal.  Indeed, suppose there are two vertices $u$ and $v$ with $d_v \ge d_u + 2$.  Let $w$ be a neighbour of $v$, and obtain a multigraph $\cG'$ by exchanging the edge $\{v,w\}$ with the edge $\{u,w\}$.  We have $d_v' = d_v - 1$, $d_u' = d_u + 1$, and $d_x' = d_x$ for all other vertices $x$.  Thus we have
\[ \sum_{x \sim y \in \cG'} (d_x' + d_y') - \sum_{x \sim y \in \cG} (d_x + d_y) = d_u + (d_u + 1) - (d_v - 1) - d_v =  2(d_u - d_v + 1) < 0,\]
and so $\sum_{x \sim y \in \cG'} d_x' + d_y' < k e(\cG')$.  Hence we may assume that all vertices in $\cG$ have degrees either $d$ or $d-1$, where $d$ is the maximum degree of $\cG$.

If $d \le k/2 \in \mathbb{Z}$, then we are done, as
\[ e(\cG) = \frac12 \sum_v d_v \le \frac{nd}{2} \le \frac{nk}{4} = \frac{n}{k} \floor{\frac{k^2}{4}}. \]

On the other hand, we must have $d \le (k+1)/2$.  Indeed, suppose $d \ge k/2+1$.  Then for every edge $u \sim v$, we have $d_u + d_v \in \{ 2d - 2, 2d-1, 2d \}$.  Since $2d - 2 \ge k$, and we must have some edges of weight at least $2d - 1 > k$ (since there is a vertex of degree $d$), we cannot have $\sum_{u \sim v} (d_u + d_v) \le k e(\cG)$.

Hence it only remains to consider the case $d = (k+1)/2$.  Let $U = \{u : d_u = (k-1)/2 \}$ and $V = \{v : d_v = (k+1)/2 \}$.  Note that edges within $U$ have weight $k-1$, edges between $U$ and $V$ have weight $k$, and edges within $V$ have weight $k+1$.  Since the average edge weight is at most $k$, we must have $e(U) \ge e(V)$.  Moreover, we have
\[ e(U,V) = \sum_{u \in U} d_u|_V = \sum_{u \in U} (d_u - d_u|_U) = \sum_{u \in U} d_u - 2e(U) = \frac{k-1}{2}\card{U} - 2e(U). \]
Similar calculations give $e(U,V) = \frac{k+1}{2} \card{V} - 2e(V)$.

Equating the two expressions, and noting that $\card{U} = n - \card{V}$, gives $\frac{k+1}{2} \card{V} = \frac{k-1}{2} \card{U} - 2e(U) + 2e(V) \le \frac{k-1}{2} (n - \card{V})$, and so $\card{V} \le \frac{n(k-1)}{2k}$.

Finally, we have
\[ e(\cG) = \frac12 \sum_x d_x = \frac12 \left[ \frac{k+1}{2} \card{V} + \frac{k-1}{2} (n - \card{V}) \right] = \frac12 \left[ \card{V} + \frac{n(k-1)}{2} \right] \le \frac{n}{k} \cdot \frac{k-1}{2} \cdot \frac{k+1}{2} = \frac{n}{k} \floor{ \frac{k^2}{4} }. \]
\end{proof}

\section{A sharp result when $k=1$} \label{sec:kequalsone}

Having resolved the case $\lambda = 0$, we now turn our attention to determining $f(n,1,\lambda)$ for arbitrary $\lambda$.

\subsection{Preliminaries} \label{sec:k1prelim}

Recall that Vu \cite{v99} proved $f(n,1,\lambda) \le 2n - 2$, with equality if and only if $\lambda = n/4$ and an $n \times n$ Hadamard matrix exists.  An $n \times n$ Hadamard matrix $H$ is an orthogonal $\{\pm 1\}$-matrix, where we may assume the last column has every entry equal to $1$.  For each of the other $n-1$ columns, define a set $F_j$ by taking $i \in F_j$ if and only if $H_{i,j} = 1$.  Add to these sets their complements, resulting in a family $\cF$ of $2n-2$ sets.  The orthogonality of the matrix $H$ ensures $\cF$ is a $1$-almost $(n/4)$-Fisher family.

We seek to sharpen this result by studying the dependence of $f(n,1,\lambda)$ on the parameter $\lambda$.  Our goal is the following theorem, which, modulo the existence of either near-optimal $\lambda$-Fisher families or appropriate Hadamard matrices, gives essentially tight bounds for all $\lambda$.  In particular, it shows that for $\lambda$ far from $n/4$, $1$-almost $\lambda$-Fisher families must have far fewer than $2n-2$ sets.

\setcounter{section}{1}
\setcounter{thm}{4}

\begin{thm}
For integers $n \ge 1$ and $\lambda \ge 0$, we have
\[ f(n,1,\lambda) \le \max \left\{ n+2, 8 \min \left\{ \lambda, \frac{n-\lambda}{3} \right\} + o ( \lambda ) \right\}. \]
\end{thm}

\setcounter{section}{3}
\setcounter{thm}{0}

To see that these bounds are essentially best possible, observe that when $\lambda < n/8$ or $\lambda > 5n/8$, this upper bound reduces to $f(n,1,\lambda) \le n + o(n)$, and there might already be Fisher families of this size.  If $n/8 \le \lambda \le n/4$, we have the upper bound $f(n,1,\lambda) \le 8 \lambda + o(\lambda)$.  If a $(4 \lambda) \times (4 \lambda)$ Hadamard matrix exists, then we can use the Hadamard construction to obtain a $1$-almost $\lambda$-Fisher with $8\lambda - 2$ sets over $[4 \lambda] \subset [n]$.  Finally, for $n/4 \le \lambda \le 5n/8$, the bound is $8 (n- \lambda) /3 + o(\lambda)$.  This bound can be achieved by taking the Hadamard construction for a $1$-almost $( (n-\lambda)/3 )$-Fisher family over $[4(n- \lambda)/3]$, and then adjoining the $\lambda - (n - \lambda)/3$ elements from $[n]\setminus [4(n-\lambda)/3]$ to each set to make the resulting family $1$-almost $\lambda$-Fisher.

To prove this theorem, we shall require two lemmas.  The first, relating the ranks of the incidence and intersection matrices, appears as Lemma 4.2 in \cite{v99}.  Recall that for a family $\cF = \{ F_1, F_2, \hdots, F_m \}$ over $[n]$, the \emph{incidence matrix} $A = A(\cF)$ is an $n \times m$ $\{0,1\}$-matrix whose columns are the characteristic vectors $\{v_1, v_2, \hdots, v_m\}$ of the sets of $\mathcal{F}$.  The \emph{intersection matrix} $M = M(\cF)$ is an $m \times m$ matrix given by $M = A^T A - \lambda J_m$, where $J_m$ is the $m \times m$ matrix with every entry equal to $1$.  Note that $M_{i,j} = \card{F_i \cap F_j} - \lambda$ for every $i,j\in [m]$.

\begin{lemma} \label{lemma:vurank}
For any set family $\cF$, $\rank (A) - 1 \le \rank (M) \le \rank (A) + 1$.  We have equality in the second inequality only if the all-$1$ vector is spanned by the columns of $M$.
\end{lemma}

\begin{proof}
The inequalities follow immediately from the subadditivity of the ranks of matrices and the fact that $\rank(A^T A) = \rank(A)$.  The necessary condition is proven in \cite{v99}.
\end{proof}

The second result is the Plotkin bound \cite{p60} from coding theory, bounding the number of codewords in a binary code with large average distance.  We express it below in the language of sets.

\begin{lemma}\label{lemma:distances}
Suppose we have a sequence of $m$ subsets $F_1, F_2, \hdots, F_m$ of $[n]$ such that $\sum_{i < j} \card{F_i \Delta F_j} \ge \binom{m}{2} \left( \frac{n}{2} + \delta \right)$ for some $\delta>0$. Then $m \le \frac{n}{2 \delta} + 1$.
\end{lemma}

\begin{proof}
By double-counting, $\sum_{i < j} \card{F_i \Delta F_j} = \sum_{x \in [n]} \card{\cF(x)} \card{\cF \setminus \cF(x)}$.  Each summand on the right-hand side is at most $\frac{m^2}{4}$, and hence $\binom{m}{2} \left( \frac{n}{2} + \delta \right) \le \frac{nm^2}{4}$.  Solving for $m$ gives the desired inequality.
\end{proof}

Armed with these lemmas, we may proceed to prove Theorem \ref{thm:kequalsone}.

\subsection{Proof of Theorem \ref{thm:kequalsone}} \label{subsec:proofforkequals1}

Suppose $\cF$ is $1$-almost $\lambda$-Fisher.  We wish to show
\[ \card{\cF} \le \max \left\{ n + 2, 8 \min \left\{ \lambda, \frac{n - \lambda}{3} \right\} + o ( \lambda ) \right\}. \]

First note that if $\lambda = 0$, then by Theorem \ref{thm:almostdisjoint} we have $\card{\cF} \le n + 1$, and so we may assume $\lambda \ge 1$.  Now suppose there is some $F \in \cF$ with $\card{F} = \lambda$.  Since $\cF$ is $1$-almost $\lambda$-Fisher, there may be some other set $F'$ such that $\card{F \cap F'} \neq \lambda$.  However, every set in $\cF \setminus \{F'\}$ must intersect $F$ in $\lambda$ elements, and must therefore contain $F$.  As these sets all intersect in the $\lambda$ elements of $F$, it follows that the restriction of $\cF \setminus \{F'\}$ to the universe $[n]\setminus F$ is a $1$-almost disjoint family over the $n - \lambda$ elements in $[n] \setminus F$.  Applying Theorem \ref{thm:almostdisjoint}, it follows that $\card{\cF \setminus \{F'\}} \le n - \lambda + 1 \le n$, and so $\card{\cF} \le n + 1$, as required.  Thus we may further assume $\card{F} > \lambda$ for all $F \in \cF$.

Next we consider the structure of the auxiliary graph $G = G(\cF)$.  Since $\cF$ is $1$-almost $\lambda$-Fisher, $G$ has maximum degree $1$.  Hence the connected components of $G$ are either isolated edges, which we denote as copies of $P_1$, or isolated vertices.  Let $\cF' = \{F_1, F_2\} \subset \cF$ be an edge in $G$.  We then have 
\[ M(\cF') = \begin{pmatrix}
	\card{F_1} - \lambda & \card{F_1 \cap F_2} - \lambda \\
	\card{F_1 \cap F_2} - \lambda & \card{F_2} - \lambda
\end{pmatrix}. \]
Since $\card{F_1} - \lambda > 0$, it follows that $\rank( M(\cF')) \ge 1$.  Moreover, by considering the determinant of this matrix, we have $\rank(M(\cF')) = 2$ unless
\begin{equation} \label{eqn:determinant}
 \left( \card{F_1} - \lambda \right) \left( \card{F_2} - \lambda \right) = \left( \card{F_1 \cap F_2} - \lambda \right)^2,
\end{equation}
in which case $\rank(M(\cF')) = 1$. Since $\card{F_1 \cap F_2} < \max \{ \card{F_1}, \card{F_2} \}$, equality requires $\card{F_1 \cap F_2} < \lambda$.

We partition $\cF$ according to the connected components of $G$.  We may write $\cF = \cup_i \cF_i$, where $\cF_i$ is a rank-$1$ $P_1$ for $1 \le i \le r$, a rank-$2$ $P_1$ for $r + 1 \le i \le r + s$, and an isolated vertex for $r + s + 1 \le i \le r + s + t$.  The total number of sets is then given by $m = \card{\cF} = 2r + 2s + t$.

We shall prove the desired bound on $m$ through three intermediate bounds on subfamilies of $\cF$. \\

\noindent \underline{Bound I:} $m \le n + r + 1$.

\begin{proof}
Since $A = A(\cF)$ is an $n \times m$ matrix, we have $\rank(A) \le n$.  By Lemma \ref{lemma:vurank}, it follows that $\rank(M) \le \rank(A) + 1 \le n + 1$.  However, $M = M(\cF)$ is a block-diagonal matrix, with submatrices $M_i = M(\cF_i)$, $1 \le i \le r + s + t$, on the diagonal.  By our choice of partition, we have
\[ \rank(M) = \sum_i \rank(M_i) = \sum_{i=1}^r \rank(M_i) + \sum_{i = r+1}^{r+s} \rank(M_i) + \sum_{i = r+s+1}^{r+s+t} \rank(M_i) = r + 2s + t. \]
Thus $m = 2r + 2s + t = \rank(M) + r \le n + r + 1$, as claimed.
\end{proof}

Hence if $r \le 1$ we have $\card{\cF} \le n + 2$, giving the desired upper bound on the size of the family.  Thus we may assume that $\cF$ has at least two rank-$1$ $P_1$'s.  We now apply a very different argument.  We label the sets in the connected components as $\cF_i = \{ F_{i,0}, F_{i,1} \}$ for $1 \le i \le r + s$, and $\cF_i = \{F_i\}$ for $r + s + 1 \le i \le r + s + t$.

First note that we can deduce a large amount of structural information about the sets in the rank-$1$ pairs.  Let $\nu = \min_{1 \le i \le r} \card{F_{i,0} \cap F_{i,1}} < \lambda$ and $\mu = \lambda - \nu$. Without loss of generality, suppose this minimum is attained at $i = 1$, and let $V = F_{1,0} \cap F_{1,1}$ and $U=(F_{1,0}\cup F_{1,1})\setminus V$.  Consider any other set $F \in \cF$.  We have $\card{F} \ge \card{F \cap \left( F_{1,0} \cup F_{1,1} \right)} = \card{F \cap F_{1,0}} + \card{F \cap F_{1,1}} - \card{F \cap V} \ge \lambda + \lambda - \nu = 2 \lambda - \nu$.  As this holds for $F_{i,0}$ and $F_{i,1}$ for any $2 \le i \le r$, we have
\[ \left( \card{F_{i,0}} - \lambda \right) \left( \card{F_{i,1}} - \lambda \right) \ge \left( \lambda - \nu \right)^2 \ge \left( \lambda - \card{F_{i,0} \cap F_{i,1}} \right)^2, \]
since $\card{F_{i,0} \cap F_{i,1}} < \lambda$ for every $i \le r$, and $\nu$ was chosen to minimise $\card{F_{i,0} \cap F_{i,1}}$.  In order to have equality, which we require in \eqref{eqn:determinant}, we must have $V \subset F_{i,j}$ and $F_{i,j} \subset F_{1,0} \cup F_{1,1}$ for all $i \le r$ and $j \in [2]$.  Moreover, we must have $F_{i,0} \cap F_{i,1} = V$ and, by symmetry, $(F_{i,0}\cup F_{i,1})\setminus V=U$.  Hence it follows that there are disjoint sets $V$ of size $\nu$ and $U$ of size $4 \mu$ such that all sets $F_{i,j}$ in the rank-$1$ edges contain $V$, and, outside this common core, the pairs $\cF_i$ each partition $U$ into two equal parts.  We let $W = [n] \setminus (U \cup V)$ be those elements not covered by sets in rank-$1$ edges, and let $\gamma = \card{W} = n - 4\mu - \nu$ denote its size.  Using this structure, we shall first deduce a preliminary bound on the size of $\cF$. \\

\noindent \underline{Bound II:} $m \le 16 \mu$.

\begin{proof}

We study the subfamily $\cG = \{F_{i,0} : 1 \le i \le r + s \} \cup \{ F_i : r + s + 1 \le i \le r + s + t \}$.  Note that we are only taking one set from each edge, and so all pairwise intersections in $\cG$ have size $\lambda$.

We shall consider the symmetric differences between sets of $\cG$ over the ground set $U$.  For any two sets $F, F' \in \cG$, we will show
\begin{equation} \label{eqn:difference}
\card[U]{F \Delta F'} = 2 \mu + 2 \card[V]{\comp{F} \cap \comp{F'}} + 2 \card[W]{F \cap F'},
\end{equation}
where $\comp{F}$ denotes the complement of $F$.

Start with the observation that $\card[U]{F \Delta F'} = \card[U]{F} + \card[U]{F'} - 2 \card[U]{F \cap F'}$.  There must be some $1 \le i \le r$ such that $F \notin \{F_{i,0}, F_{i,1}\}$, and so $\card{F \cap F_{i,0}} = \lambda = \card{F \cap F_{i,1}}$.  Since $V = F_{i,0} \cap F_{i,1}$ and $F_{i,0} \setminus V$ and $F_{i,1} \setminus V$ partition $U$, we obtain $\card[U]{F} = 2 \lambda - 2 \card[V]{F}$.  Similarly, $\card[U]{F'} = 2 \lambda - 2 \card[V]{F'}$.  Hence $\card[U]{F \Delta F'} = 4 \lambda - 2 \left( \card[V]{F} + \card[V]{F'} \right) - 2 \card[U]{F \cap F'}$.  For future reference, note that $\card[U]{F} = 2 \lambda - 2 \card[V]{F} \ge 2 \lambda - 2 \nu = 2 \mu$.

Now $\card[V]{F} + \card[V]{F'} = \card[V]{F \cup F'} + \card[V]{F \cap F'} = \nu - \card[V]{\comp{F} \cap \comp{F'}} + \card[V]{F \cap F'}$.  Moreover, since $\cG$ is a $\lambda$-Fisher family, we must have $\card{F \cap F'} = \card[U]{F \cap F'} + \card[V]{F \cap F'} + \card[W]{F \cap F'} = \lambda$.  Putting this all together, we have the claimed equality,
\[ \card[U]{F \Delta F'} = 2 \left( \lambda - \nu \right) + 2 \card[V]{\comp{F} \cap \comp{F'}} + 2 \left( \lambda - \card[V]{F \cap F'} - \card[U]{F \cap F'} \right) = 2 \mu + 2 \card[V]{\comp{F} \cap \comp{F'}} + 2 \card[W]{F \cap F'}. \]

In particular, the differences between sets in $\cG$ over $U$ have size at least $2 \mu$.  Our universe $U$ has size $4 \mu$, and so this is not enough to apply Lemma \ref{lemma:distances}, as we would have $\delta = 0$.  However, since each set in $\cG$ contains at least $2 \mu$ elements of $U$, we can find some $u \in U$ contained in at least half the sets of $\cG$.  Considering the sets in $\cG(u)$ over the universe $U \setminus \{u\}$, we have at least $\card{\cG} / 2$ sets with pairwise distances at least $2 \mu$ over a universe of size $4 \mu - 1$.  Taking $\delta = 1/2$ in Lemma \ref{lemma:distances}, this gives $\card{\cG}/2 \le 4 \mu$, and so $\card{\cG} = r + s + t \le 8 \mu$.  Hence $\card{\cF} = 2r + 2s + t \le 2 \card{\cG} \le 16 \mu$, as claimed.
\end{proof}

If $\mu \le n / 16$, then by Bound II we have $m \le 16 \mu \le n$, which would give the desired bound on the size of $\cF$.  Hence we may now assume $\mu \ge n / 16$ and $\card{\cF} \ge n + 2 = 4 \mu + \nu + \gamma + 2 > 4 \mu + \nu$ . \\

\noindent \underline{Bound III:} $m \le 8 \mu + o(\mu)$.

\begin{proof}

To complete the proof, we will now remove the extra factor of two from Bound II.  We will also show that when $\cF$ only consists of rank-$1$ $P_1$'s, then we have the more precise bound $m \le 8 \mu$, a result that will be used in Lemma \ref{lemma:boundingthenumberoflowrankstructures}.

Let $\cG$ be as above.  In Bound II, we used \eqref{eqn:difference} to bound $\card[U]{F \Delta F'}$ from below by $2 \mu$.  We will now show that the two additional terms cannot contribute too much, as otherwise we will gain in the application of Lemma \ref{lemma:distances}.

We define some additional notation to keep track of the sets and pairs of sets that contribute to these additional terms.  Let $\cP^{(i)} = \left\{ \{F,F'\} \subset \cG : 2^i \le \card[V]{\comp{F} \cap \comp{F'}} < 2^{i+1} \right\}$ be those pairs that fail to cover between $2^i$ and $2^{i+1}$ elements of $V$, and let $\cR^{(i)} = \cup_{\{F,F'\} \in \cP^{(i)}} \{F,F'\}$ be the sets involved in such pairs.  Similarly, let $\cQ^{(i)} = \left\{ \{F,F'\} \subset \cG : 2^i \le \card[W]{F \cap F'} < 2^{i+1} \right\}$ be the pairs that intersect in between $2^i$ and $2^{i+1}$ elements in $W = [n] \setminus (U \cup V)$, and let $\cS^{(i)} = \cup_{\{F,F'\} \in \cQ^{(i)}} \{F,F'\}$ be the sets themselves.  We denote the sizes of these families by $p_i = \card{\cP^{(i)}}$, $q_i = \card{\cQ^{(i)}}$, $r_i = \card{\cR^{(i)}}$ and $s_i = \card{\cS^{(i)}}$.  The following lemma bounds these quantities.

\begin{lemma} \label{lem:addterms}
We must have $p_i \le \mu r_i 2^{-(i+1)}$, $r_i 2^i \le 4 \mu \sqrt{2 \nu}$, $q_i \le \mu s_i 2^{-(i+1)}$ and $s_i 2^i \le 4 \mu \sqrt{2 \gamma}$.
\end{lemma}

Given these bounds, it follows that the sets in $\cG$ are similar in structure to those from the Hadamard construction.  Inspired by that example, we introduce complements with respect to $U$ to remove the extra factor of two in Bound II.  Let $\cH = \{ F \cap U : F \in \cG \} \cup \{ \comp{F} \cap U : F \in \cG \}$, giving $2 \card{\cG}$ sets (counted with multiplicity, if needed).  

We choose an element $x \in U$ uniformly at random, and apply Lemma \ref{lemma:distances} to $\cH(x)$.  First note that since $\cH$ consists of pairs of sets and their complements, $\cH(x)$ contains one set from each pair, and so $\card{\cH(x)} = \card{\cG}$ for all $x\in U$.  Now consider the expected sum of differences between pairs of sets in $\cH(x)$.  A pair of sets $\{F, F' \} \subset \cG$ gives rise to four pairs of sets to consider in $\cH(x)$: $\{F,F'\}, \{F,\comp{F'}\}, \{\comp{F},F'\}$ and $\{\comp{F},\comp{F'}\}$ (note that a set and its complement cannot both be in $\cH(x)$, and hence contribute nothing to the expectation).  Any pair of sets $\{S,T\}$ is contained in $\cH(x)$ with probability $\card[U]{S \cap T}/4 \mu$, and has difference $\card[U]{S \Delta T}$, thus contributing $\card[U]{S \cap T} \card[U]{S \Delta T}/4 \mu$ to the expectation.

By summing the corresponding terms for the four pairs associated with $\{F, F' \} \subset \cG$
and using (\ref{eqn:difference}), we find that the contribution to the expectation is 
\begin{align*}
	\frac{1}{2 \mu} \card[U]{F \Delta F'} &\left( 4 \mu - \card[U] { F \Delta F'} \right) = \frac{1}{2 \mu} \left( 2 \mu + 2 \card[V]{\comp{F} \cap \comp{F'}} + 2 \card[W]{F \cap F'} \right) \left( 2 \mu - 2 \card[V]{\comp{F} \cap \comp{F'}} - 2 \card[W]{F \cap F'} \right) \\
	&= \frac{2}{\mu} \left( \mu^2 - \left( \card[V]{\comp{F} \cap \comp{F'}} + \card[W]{F \cap F'} \right)^2 \right) \ge \frac{2}{\mu} \left( \mu^2 - 2 \card[V]{\comp{F} \cap \comp{F'}}^2 - 2 \card[W]{F \cap F'}^2 \right).
\end{align*}

Hence, summing over all pairs, the expected sum of differences in $\cH(x)$ is at least
\begin{align*}
	\sum_{F, F' \in \cG} \frac{2}{\mu}& \left( \mu^2 - 2 \card[V]{\comp{F} \cap \comp{F'}}^2 - 2 \card[W]{F \cap F'}^2 \right) = 2 \mu \binom{\card{\cG}}{2} - \frac{4}{\mu} \sum_{F,F' \in \cG} \card[V]{\comp{F} \cap \comp{F'}}^2 - \frac{4}{\mu} \sum_{F,F' \in \cG} \card[W]{F \cap F'}^2 \\
	&\ge 2 \mu \binom{\card{\cG}}{2} - \frac{4}{\mu} \sum_{i = 0}^{\log \nu} p_i (2^{i+1})^2 - \frac{4}{\mu} \sum_{i = 0}^{\log \gamma} q_i (2^{i+1})^2 \\
	&\ge 2 \mu \binom{\card{\cG}}{2} - \frac{4}{\mu} \sum_{i = 0}^{\log \nu} \left( \mu r_i 2^{-(i+1)} \right) ( 2^{i+1} )^2 - \frac{4}{\mu} \sum_{i = 0}^{\log \gamma} ( \mu s_i 2^{-(i+1)} ) (2^{i+1})^2 \\
	&\ge 2 \mu \binom{\card{\cG}}{2} - 8 \sum_{i = 0}^{\log \nu} r_i 2^i - 8 \sum_{i = 0}^{\log \gamma} s_i 2^i \ge 2 \mu \binom{\card{\cG}}{2} - 64 \mu \nu^{1/2} \log \nu - 64 \mu \gamma^{1/2} \log \gamma,
\end{align*}
where in the third and fifth inequalities we used the bounds on $p_i, q_i, r_i$ and $s_i$ given by Lemma \ref{lem:addterms}.

There is thus some choice of $x \in U$ for which the sum of differences in $\cH(x)$ is at least this expectation.  Removing the common element $x$ from the $\card{\cG}$ sets in $\cH(x)$ does not affect the differences, and reduces the size of the universe to $4 \mu - 1$, and hence in the application of Lemma \ref{lemma:distances} we may take
\[ \delta \ge \frac12 - \frac{64 \mu \left( \nu^{1/2} \log \nu + \gamma^{1/2} \log \gamma \right)}{\binom{\card{\cG}}{2}} > \frac12 - \frac{32 \left( \nu^{1/2} \log \nu + \gamma^{1/2} \log \gamma \right)}{\mu}, \]
recalling $\card{\cG} \ge \frac12 \card{\cF} \ge \frac12 \left( n + 2 \right) \ge 2 \mu + 1$. As $\nu, \gamma \le n$ and $\mu \ge n/16$, this lower bound is $\frac12 - o(1)$.  Hence, provided $n$ is sufficiently large, Lemma \ref{lemma:distances} gives
\[ \card{\cG} \le \frac{4\mu - 1}{2\left( \frac12 - o(1) \right)} + 1 =4 \mu + o(\mu). \]
Thus, $\card{\cF} \le 2 \card{\cG} \le 8 \mu + o( \mu )$, as claimed.

We remark that careful calculation shows the $o(\mu)$ error term is in fact $O \left( \nu^{1/2} \log \nu + \gamma^{1/2} \log \gamma \right)$, which is at most $O \left( n^{1/2} \log n \right)$.  Moreover, if the original family of sets $\cF$ consists only of rank-$1$ $P_1$'s, then we do not require the error term at all.  Indeed, for sets in rank-$1$ $P_1$'s, we must have $V \subset F \subset U \cup V$, and hence there are no additional terms in \eqref{eqn:difference}.  Thus in this setting we have $p_i = q_i = r_i = s_i = 0$ for all $i$, resulting in the bound $m \le 8 \mu$.

\end{proof}

In order to maximise this bound, then, we seek to maximise $\mu$.  We have two constraints: $\mu \le \lambda$, and $4 \mu + \nu \le n$, where $\nu = \lambda - \mu$.  Together, this gives $\mu \le \min \left\{ \lambda, \frac{ n - \lambda }{3} \right\}$.  Hence, combining Bounds I and III, we find $\card{\cF} \le \max \left\{ n + 2 , 8 \min \left\{ \lambda, \frac{n - \lambda}{3} \right\} + o(\lambda) \right\}$.  As this holds for all $1$-almost $\lambda$-Fisher families, Theorem \ref{thm:kequalsone} follows, pending the proof of Lemma \ref{lem:addterms}.

\begin{proof}[Proof of Lemma \ref{lem:addterms}]
We first obtain the bound on $p_i$, showing the pairs in $\cP^{(i)}$ cannot be too dense with respect to $\cR^{(i)}$.  Considering the total distances over $U$ between pairs of sets in $\cR^{(i)}$, we have
\[ \sum_{F,F' \in \cR^{(i)}} \card[U]{F \Delta F'} = \sum_{F, F' \in \cR^{(i)}} \left[ 2 \mu + 2 \card[V]{\comp{F} \cap \comp{F'}} + 2 \card[W]{F \cap F'} \right] \ge 2 \mu \binom{r_i}{2} + 2 p_i 2^i, \]
using \eqref{eqn:difference} and noting that there are at least $p_i$ pairs for which we have $\card[V]{\comp{F} \cap \comp{F'}} \ge 2^i$.  If $p_i \ge 1$ then $r_i \ge 2$, and taking $\delta = 2 p_i 2^i / \binom{r_i}{2}$ in Lemma \ref{lemma:distances} gives $r_i \le 4 \mu / 2 \delta + 1$, which simplifies to $p_i \le \mu r_i 2^{-(i+1)}$.  Running the same argument with $\cQ^{(i)}$ and $\cS^{(i)}$ gives $q_i \le \mu s_i 2^{-(i+1)}$.

We now show that there cannot be too many sets in $\cR^{(i)}$, for otherwise these additional terms in \eqref{eqn:difference} would contribute too much when applying Lemma \ref{lemma:distances}.  Note that if $\nu = 0$, we must have $\cR^{(i)} = \emptyset$, and so we would be done.  Thus we may assume $\nu \ge 1$.  Since a set $F \in \cR^{(i)}$ comes from some pair $\{F,F'\} \in \cP^{(i)}$, we must have $\card[V]{\comp{F}} \ge 2^i$
and therefore $\sum_{F \in \cR^{(i)}} \card[V]{\comp{F}} \ge r_i2^i $.  On average a vertex $x \in V$ is contained in at least $r_i2^i/\nu$ sets $\comp{F}$, where $F \in \cR^{(i)}$. Hence, by a standard application of Jensen's inequality, we  have
\[ \sum_{F,F' \in \cR^{(i)}} \card[V]{\comp{F} \cap \comp{F'}} 
\ge \nu \binom{\frac{r_i 2^i}{\nu}}{2} = r_i 2^{i-1} \left( \frac{r_i 2^i}{\nu} - 1 \right). \]
If $r_i 2^i \le 2 \nu$, we certainly have $r_i 2^i \le 4 \mu \sqrt{2 \nu}$ (since $\nu \le n$, $\mu = \Omega(n)$, and $n$ is large).  Otherwise the above quantity is at least $r_i^2 2^{2i - 2} / \nu$.  This term is a lower bound on the sum of the additional terms we obtain in \eqref{eqn:difference} when summing $\card[U]{F \Delta F'}$ over the sets in $\cG$. Taking $\delta = 2 r_i^2 2^{2i - 2} / ( \nu \binom{\card{\cG}}{2} )$ in Lemma \ref{lemma:distances} gives $\card{\cG} \le 4 \mu / (2\delta) + 1 = 2 \mu \nu \card{\cG} (\card{\cG}-1) / r_i^2 2^{2i} + 1$ which simplifies to $\card{\cG} \ge r_i^22^{2i}/(2\mu\nu)$. If $r_i 2^i > 4 \mu \sqrt{2 \nu}$, we then have $\card{\cG} > 16\mu$, contradicting Bound II. Hence we may assume $r_i 2^i \le 4 \mu \sqrt{2 \nu}$ for every $i$.

We run a similar argument to bound $s_i$, assuming $s_i\neq 0$ (and so $\gamma\neq 0$).  If $F \in \cS^{(i)}$, then there is some $F' \in \cS^{(i)}$ such that $\card[W]{F \cap F'} \ge 2^i$.  In particular, $\card[W]{F} \ge 2^i$.  Jensen then gives
\[ \sum_{F,F' \in \cS^{(i)}} \card[W]{ F \cap F'} \ge \gamma \binom{\frac{s_i 2^i}{\gamma}}{2} \ge \frac{s_i^2 2^{2i - 2}}{\gamma}, \]
provided $s_i 2^i \ge 2 \gamma$.  Otherwise we have $s_i 2^i \le 2 \gamma \le 4 \mu \sqrt{2 \gamma}$ (since $\mu \ge n / 16$, $\gamma \le n$ and $n$ is large).  Again, this quantity lower bounds the additional terms in \eqref{eqn:difference}.  As before, we may apply Lemma \ref{lemma:distances} and Bound II to deduce that $s_i 2^i \le 4 \mu \sqrt{2 \gamma}$ for every $i$.
\end{proof}

\section{Tight bounds for $k=2$} \label{sec:kequalstwo}

We now study the problem for $k = 2$, with the goal of proving Theorem \ref{thm:kequalstwo}, reproduced below.

\setcounter{section}{1}
\setcounter{thm}{5}

\begin{thm}
For $n$ sufficiently large and $0 \le \lambda \le n$, we have the bounds
\begin{itemize}
	\item[(i)] $f(n,2,\lambda) \le 2n - 2$.
	\item[(ii)] $f(n,2,\lambda) \le \frac13 \left( 5n + 4 \min \left\{ \lambda, \frac{n - \lambda}{3} \right\} + 7 \right)$.
	\item[(iii)] $f(n,2,\lambda) \le \left( \frac32 + o(1) \right) n$ when $\lambda = o(n)$.
\end{itemize}
\end{thm}

\setcounter{section}{4}
\setcounter{thm}{0}

The bound in part (i) is best possible, as shown by the Hadamard construction described in Section \ref{sec:k1prelim}.  This shows that, surprisingly, allowing one extra bad intersection per set does not provide sufficient freedom to construct larger families.  Part (ii) is a stability result, showing that there only exist $2$-almost $\lambda$-Fisher families of size close to $2n$ when $\lambda$ is close to $n/4$; however, we do not believe these bounds to be tight.  Thus in part (iii) we provide a sharper bound when $\lambda = o(n)$, where the constant $\frac32$ cannot be improved in light of Theorem \ref{thm:almostdisjoint}.

The proof of part (i) is an extension of the method of Vu \cite{v99}, but the proofs of parts (ii) and (iii) use a combination of these ideas and our arguments from Section \ref{sec:kequalsone}.  These proofs are given in Section \ref{sec:proofofk2}.  We begin, though, by providing some necessary lemmas in Section \ref{sec:lemmasfork2}.

\subsection{Preliminary Lemmas} \label{sec:lemmasfork2}

These simple lemmas, whose proofs we give in Appendix \ref{app:appendix}, will allow us to control the ranks of matrices appearing in the proof of Theorem \ref{thm:kequalstwo}. The first lemma shows that, under some mild conditions, we can always find a number of linearly independent vectors in various sets of $\{0,1\}$-vectors.

\begin{lemma}
\label{lem:(0,1)-vectors}
Let $v_i\in \{0,1\}^n$, $i\in [5]$, be five distinct non-zero vectors. Suppose that there exist $\lambda\in \RR\setminus\{0\}$ and $v\in \RR^n$ such that $v\cdot v_i=\lambda$ for $i\in [5]$. Then:
\begin{enumerate}[(a)]
\item The vectors $\{ v_i \}_{i \in [3]}$ are linearly independent.
\item The vectors $v_1-v_2$ and $v_1-v_3$ are linearly independent.
\item If the vectors $\{v_i\}_{i\in [4]}$ are linearly dependent, then $v_1+v_2=v_3+v_4$ holds for some relabelling of these four vectors.
\item Four of the vectors $\{v_i\}_{i\in [5]}$ are linearly independent.
\end{enumerate}
\end{lemma}

Let $\cF = \{F_1, \hdots, F_m\}$ denote a $2$-almost $\lambda$-Fisher family of sets over $[n]$ with parameter $\lambda \ne 0$. Assume for now that $m \ge 6$ and $\card{F_i} > \lambda$ for each $i\in [m]$. The next lemmas relate the auxiliary graph $G(\cF)$ with the ranks of the intersection matrix $M(\cF)$ and the element-set incidence matrix $A(\cF)$ (see section \ref{subsec:outline and notation} for the relevant definitions).

Note that, since $\cF$ is a $2$-almost $\lambda$-Fisher family, the graph $G(\cF)$ has maximum degree $2$, and hence its connected components are paths and cycles. Thus, if the sets in $\cF$ are ordered appropriately, $M(\cF)$ is a block-diagonal matrix, with each block corresponding to a path or a cycle in $G(\cF)$. Since $M(\cF)$ is block-diagonal, the rank of $M(\cF)$ is the sum of the ranks of its block matrices.  The following lemma provides lower bounds on the ranks of the corresponding subfamilies $\cF' \subset \cF$, which we identify with the components $G(\cF')$ in $G(\cF)$.

\begin{lemma} \label{lem:compranks}
The ranks of the components can be bounded as follows:
\begin{enumerate}[(a)]
\item If $\cF'$ is the $s$-vertex path $P_{s-1}$, $\rank(M(\cF')) \ge s-1$.
\item If $\cF'$ is the $s$-vertex cycle $C_s$, $\rank(M(\cF')) \ge s-2$.
\item If $\cF'$ is the triangle $C_3$, and there is some set $F$ whose intersections with every set in $\cF'$ all have size $\lambda$, then $\rank(M(\cF')) \ge 2$.
\end{enumerate}
\end{lemma}

The rank-$1$ $P_1$'s and rank-$2$ $C_4$'s will play an important role in our proofs, and so we obtain some further information about the corresponding matrices in the following lemma.

\begin{lemma} \label{lem:lowrankspan}
Let $\cF$ be a $2$-almost $\lambda$-Fisher family of size $\card{\cF} \ge 5$ with $\card{F} > \lambda$ for all $F \in \cF$.  If a component $\cF' \subset \cF$ is either a rank-$1$ $P_1$ or a rank-$2$ $C_4$, then the columns of $M(\cF')$ do not span the all-$1$ vector $(1, \hdots, 1)^T$.
\end{lemma}

\subsection{Proof of Theorem \ref{thm:kequalstwo}} \label{sec:proofofk2}

With these lemmas in place, we may now proceed with the proof of Theorem \ref{thm:kequalstwo}.

Let $\cF = \{F_1, F_2, \hdots, F_m\}$ be a $2$-almost $\lambda$-Fisher family over $[n]$, for some large enough $n$.  Note that if $\lambda = 0$, then by Theorem \ref{thm:almostdisjoint}, it follows that $\card{\cF} \le 3n/2 + 1$, which is small enough to satisfy the bounds from all three parts.  Hence we may assume $\lambda \ge 1$.  If there is some set $F \in \cF$ with $\card{F} = \lambda$, then we know at most two sets fail to contain $F$.  The remaining sets, restricted to the universe $[n]\setminus F$, form a $2$-almost disjoint family $\cF'$.  It follows from Theorem \ref{thm:almostdisjoint} that
\[\card{\cF} \le \card{\cF'} + 2 \le \frac32 (n - \lambda) + 1 + 2 \le \frac32 n + \frac32, \]
which again suffices.  Finally, observe that we may take $m \ge 6$, as otherwise there is nothing to prove for large $n$.  Hence we may assume $m \ge 6, \lambda \ge 1$ and $\card{F} > \lambda$ for all $F \in \cF$, and thus the lemmas of the previous subsection apply.

As mentioned before, we can order the sets in $\cF$ in such a way that the intersection matrix $M(\cF)$ is a block-diagonal matrix, with each block corresponding to a connected component - path or cycle - in $G(\cF)$. Let $M_1, \hdots, M_t$ be the blocks of $M(\cF)$ and let $m_i$ be the number of sets in the corresponding component.  Since $A(\cF)$ is an $n \times m$ matrix, $\rank(A(\cF)) \le n$.  By Lemma \ref{lemma:vurank}, we then have
\[ \sum_{i=1}^t \rank (M_i) = \rank (M(\cF)) \le \rank (A(\cF)) + 1 \le n + 1. \]

Moreover, by Lemma \ref{lem:compranks}, it follows that $\rank (M_i) \ge \frac23 m_i$, unless $M_i$ corresponds to a rank-$1$ $P_1$, a rank-$2$ $C_4$ or a rank-$3$ $C_5$.  Suppose the first $p$ blocks are rank-$1$ $P_1$'s, the next $q$ blocks are rank-$2$ $C_4$'s and the following $r$ blocks are rank-$3$ $C_5$'s.  If we separate the rank-$1$ $P_1$'s and the rank-$2$ $C_4$'s, then the remaining blocks $M_i$ have rank at least $\frac35 m_i$, giving
\[ n + 1 \ge \rank (M(\cF)) = \sum_i \rank (M_i) \ge p + 2q + \frac35 \sum_{i > p + q} m_i = p + 2q + \frac35 (m - 2p - 4q), \]
resulting in
\begin{equation} \label{eqn:withoutC5}
m \le \frac53n + \frac{p}{3} + \frac{2q}{3} + \frac53.
\end{equation}

To obtain a sharper bound, we must account for the number of rank-$3$ $C_5$'s as well.  We have
\[ n + 1 \ge \sum_i \rank (M_i) \ge p + 2q + 3r + \frac23 \sum_{i > p + q + r} m_i = p + 2q + 3r + \frac23 (m - 2p - 4q - 5r), \]
and so
\begin{equation} \label{eqn:withC5}
m \le \frac32 n + \frac{p}{2} + q + \frac{r}{2} + \frac32.
\end{equation}

By bounding $p, q$ and $r$ appropriately in the following subsections, we shall establish the three bounds in Theorem \ref{thm:kequalstwo}.

\subsubsection{Proof of part (i)} \label{sec:k2i}

From \eqref{eqn:withoutC5}, if there are no rank-$1$ $P_1$'s or rank-$2$ $C_4$'s (so $p = q = 0$), we then have $m \le \frac53 n + \frac53 < 2n - 2$ for $n$ large.  Hence we may assume there is either a rank-$1$ $P_1$ or a rank-$2$ $C_4$. In either case, by Lemma \ref{lem:lowrankspan} it follows that the columns of $M(\cF)$ do not span the all-$1$ vector.  By Lemma \ref{lemma:vurank}, we then in fact have $\rank (M(\cF)) \le \rank (A(\cF)) \le n$.  Thus, if $\rank (M(\cF)) \ge \frac12 m + 1$, we would have $m \le 2n - 2$, as desired. Otherwise, since $\rank(M_i) \ge \frac12 m_i + \frac12$ for any block $M_i$ that is not a rank-$1$ $P_1$ or rank-$2$ $C_4$, if $\rank (M(\cF))< \frac12 m+1$ we must have one of the following cases:
\begin{itemize}
	\item[I.] $\rank (M(\cF)) = \frac12 m$, and each block $M_i$ is a rank-$1$ $P_1$ or a rank-$2$ $C_4$, or
	\item[II.] $\rank (M(\cF)) = \frac12 m + \frac12$, there is one block $M_t$ of rank $\frac12 m_t + \frac12$, and the remaining blocks are rank-$1$ $P_1$'s or rank-$2$ $C_4$'s.	
\end{itemize}

We shall now consider the characteristic vectors of the sets in $\cF$.  For $1 \le i \le m$, let $v_i$ be the characteristic vector of the set $F_i$. \\

\noindent \underline{Case I:}  For $2 \le i \le t$, we assign to each block $M_i$ a set $X_i$ of $\frac12 m_i$ vectors as follows.  If $M_i$ is a rank-$1$ $P_1$ with sets $F_{j_1}$ and $F_{j_2}$, we set $X_i = \{v_{j_1} - v_{j_2}\}$.  If $M_i$ is a rank-$2$ $C_4$ with sets $F_{j_1}, F_{j_2}, F_{j_3}$ and $F_{j_4}$, we set $X_i = \{v_{j_1} - v_{j_2}, v_{j_1} - v_{j_3}\}$. By Lemma \ref{lem:(0,1)-vectors} (b), these two vectors are linearly independent.

To the block $M_1$, we assign a set $X_1$ of $\frac12 m_1 + 1$ vectors.  If $M_1$ is a rank-$1$ $P_1$ with sets $F_{j_1}$ and $F_{j_2}$, set $X_1 = \{v_{j_1}, v_{j_2}\}$.  If $M_1$ is a rank-$2$ $C_4$ with sets $F_{j_1}, F_{j_2}, F_{j_3}$ and $F_{j_4}$, we take $X_1 = \{ v_{j_1}, v_{j_2}, v_{j_3}\}$.  By Lemma \ref{lem:(0,1)-vectors} (a), $X_1$ is also linearly independent.

Since sets from different blocks have pairwise intersections of size $\lambda$, it is easy to see that vectors in $X = \cup_i X_i$ from different blocks are orthogonal to one another.  Thus $X$ is a collection of $\frac12 m + 1$ linearly independent vectors in $\mathbb{R}^n$, and so $\frac12 m + 1 \le n$, giving $m \le 2n - 2$. \\

\noindent \underline{Case II:} For $1 \le i \le t-1$, we define sets of vectors $X_i$ as before.

In light of Lemma \ref{lem:compranks}, the block $M_t$ of rank $\frac12 m_t + \frac12$ must be either a rank-$2$ $P_2$, a rank-$2$ $C_3$, or a rank-$3$ $C_5$.  If the block $M_t$ is either a $P_2$ or a $C_3$, then by Lemma \ref{lem:(0,1)-vectors} (a) the characteristic vectors of the three sets involved must be linearly independent, and so we take all three vectors in $X_t$.  If it is a $C_5$, then by Lemma \ref{lem:(0,1)-vectors} (d) four of the sets have linearly independent characteristic vectors, which we add to $X_t$.

As in Case I, $X = \cup_i X_i$ forms a collection of $\frac12 m + \frac32$ linearly independent vectors in $\mathbb{R}^n$, and so in this case we have $m \le 2n - 3$.

Hence, in either case we have $m \le 2n - 2$, completing the proof of part (i).

\subsubsection{Proof of part (ii)} \label{sec:k2ii}

In light of \eqref{eqn:withoutC5}, to bound the size of $2$-almost $\lambda$-Fisher families, it suffices to bound the number of rank-$1$ $P_1$'s and rank-$2$ $C_4$'s.  This is done in the following lemma.

\begin{lemma}
\label{lemma:boundingthenumberoflowrankstructures}
Let $\cF$ be a $2$-almost $\lambda$-Fisher family over $[n]$ with $\card{\cF} \ge 6$ and $\card{F} > \lambda \ge 1$ for all $F \in \cF$. Then:
\begin{enumerate}[(a)]
\item there are at most $4 \min \left\{ \lambda, \frac{n - \lambda}{3} \right\}$ rank-$1$ $P_1$'s in $G(\cF)$.
\item there is at most one rank-$2$ $C_4$ in $G(\cF)$.
\end{enumerate}
\end{lemma}

By Lemma \ref{lemma:boundingthenumberoflowrankstructures}, the number of rank-$1$ $P_1$'s is bounded by $p \le 4 \min \left\{ \lambda, \frac{n - \lambda}{3} \right\}$, while the number of rank-$2$ $C_4$'s satisfies $q \le 1$.  Substituting these bounds into \eqref{eqn:withoutC5} gives
\[ m \le \frac53n + \frac{p}{3} + \frac{2q}{3} + \frac53 \le \frac13 \left( 5n + 4 \min \left\{ \lambda, \frac{n - \lambda}{3} \right\} + 7 \right),\]
as required.  Hence we need only prove Lemma \ref{lemma:boundingthenumberoflowrankstructures} to establish the bound in part (ii).

\begin{proof}[Proof of Lemma \ref{lemma:boundingthenumberoflowrankstructures}]
We begin with part (a).  Note that by restricting ourselves to the subfamily of rank-$1$ $P_1$'s, we obtain a $1$-almost $\lambda$-Fisher family, and thus may apply our results from Section \ref{sec:kequalsone}.  If there is at most one rank-$1$ $P_1$, we are done, and hence we may assume that there are at least two.

Recall that in this case, the sets from the rank-$1$ $P_1$'s are supported on $4 \mu + ( \lambda - \mu )$ elements of $[n]$ for some $\mu \le \lambda$, and so $\mu \le \frac{n - \lambda}{3}$.  We now use Bound III from Section \ref{subsec:proofforkequals1}.  Since all the sets are in rank-$1$ $P_1$'s, we do not require the error terms that appear in Bound III.  Thus there are at most $8 \mu$ sets, and hence $4 \mu$ rank-$1$ $P_1$'s.  Given our bounds on $\mu$, it follows that there are at most $4 \min \left\{ \lambda, \frac{n - \lambda}{3} \right\}$ rank-$1$ $P_1$'s, as claimed.

Now we prove part (b). Let $F_1,F_2,F_3,F_4\in \mathcal{F}$ be sets that correspond in cyclic order to a rank-$2$ $C_4$. Define $s_i=|F_i|-\lambda$ and $p_i=|F_i\cap F_{i+1}|-\lambda$ for each $i\in [4]$ (indices considered modulo $4$). Note that for each $i\in [4]$ we are assuming that $s_i>0$ and $p_i\neq 0$. Then, we have
\[M=M(\{F_1,F_2,F_3,F_4\})=\begin{pmatrix}
s_1 & p_1 & 0 & p_4\\ 
p_1 & s_2 & p_2 & 0\\ 
0 & p_2 & s_3 & p_3\\ 
p_4 & 0 & p_3 & s_3
\end{pmatrix}.\]
It is easily seen that any two rows of $M$ are linearly independent. Moreover, since $M$ has rank $2$, any three rows of $M$ are linearly dependent. In particular, there must exist $\alpha, \beta\in \RR$ such that:
\begin{equation}
\label{eqn: rank 2 linear relation}
\alpha (s_1,p_1,0,p_4)+\beta (p_1,s_2,p_2,0)=(0,p_2,s_3,p_3)
\end{equation}

Suppose now that $p_i>0$ for all $i\in[4]$. Looking at the third and fourth coordinates in equation~(\ref{eqn: rank 2 linear relation}) it follows that $\beta=\frac{s_3}{p_2}>0$ and $\alpha=\frac{p_3}{p_4}>0$. However, looking at the first coordinate we see that we cannot have $\alpha>0$ and $\beta>0$ since in that case $\alpha s_1+\beta p_1>0$. It thus follows that at least one of the $p_i$'s is negative, and so at least one of the intersections $F_i\cap F_{i+1}$ has size less than $\lambda$.

Recall that the columns of the incidence matrix $A = A(\{F_1, F_2, F_3, F_4 \})$ are the characteristic vectors $v_i$.  Since $M=A^{T}A-\lambda J_4$ and, by Lemma \ref{lemma:vurank}, we have $\rank (A)=\rank (A^TA)\leq \rank (M)+1=3$, the vectors $\{v_i\}_{i \in [4]}$ must be linearly dependent. Thus, by Lemma \ref{lem:(0,1)-vectors} (c), one of the following linear relations must hold
\[\text{(i)}\; v_1+v_2=v_3+v_4, \;\;\text{(ii)}\;v_1+v_3=v_2+v_4, \;\;\text{(iii)}\;v_1+v_4=v_2+v_3.\]

We claim that (ii) does not hold. Indeed, if it did hold then it would follow that $F_1\cap F_3=F_2\cap F_4$.  However, since $F_1$ and $F_3$ are not adjacent in $G(\cF)$, this intersection has size $\lambda$, and so $|F_1\cap F_2\cap F_3\cap F_4|=\lambda$, contradicting the fact shown above that $|F_i\cap F_{i+1}|<\lambda$ for some $i\in [4]$.

Observe also that (i) and (iii) are the same up to a cyclic relabeling of the sets $F_i$. Thus we may assume that (i) holds. In terms of the sets $F_i$, this equation implies that they all contain a common core $X =F_1\cap F_2=F_3\cap F_4$ and are supported on the same universe $U=F_1\cup F_2=F_3\cup F_4$.

Since $|F_i\cap F_{i+1}|<\lambda$ for some $i\in [4]$ it is clear that $|X|<\lambda$ and so $p_1=p_3=|X|-\lambda <0$. Note that we cannot have both $p_2>0$ and $p_4>0$ as in that case it would follow from equation~(\ref{eqn: rank 2 linear relation}) that $\alpha=\frac{p_3}{p_4}<0$, $\beta=\frac{s_3}{p_2}>0$ and hence $\alpha s_1+\beta p_1<0$, contradicting the fact that $\alpha s_1 + \beta p_1=0$. Thus, we may assume without loss of generality that $p_2<0$. This means that $|F_2\cap F_3|=\lambda+p_2<\lambda$. Moreover, since $|F_2\cap F_4|=\lambda$ it follows that
\[|F_2|=|F_2\cap U|=|F_2\cap F_3|+|F_2\cap F_4|-|F_2\cap F_3\cap F_4|<2\lambda-|X|.\]

Let $F\in \mathcal{F}$ be a set different from all the sets $F_i$. Since $F$ is not adjacent to any of the sets $F_i$ in $G(\cF)$, $\card{F \cap F_i} = \lambda$ for all $i$, and so
\[|F|\ge |F\cap U|=|F\cap F_1|+|F\cap F_2|-|F\cap X|\ge 2\lambda-|X|.\]

We have shown that among any four sets $F_1,F_2,F_3,F_4$ in a rank-$2$ $C_4$ there is one of size strictly less than $2\lambda-|X|$, where $X=F_1\cap F_2\cap F_3\cap F_4$ is the corresponding common core. Moreover, any other set in $\mathcal{F}$ must have size at least $2\lambda-|X|$. Suppose now that there are two rank-$2$ $C_4$'s in $G(\cF)$. Let $X$ and $X'$ be the common cores of their sets and assume without loss of generality that $|X|\ge |X'|$. By the above there is a set $F$ in the first $C_4$ of size less than $2\lambda- |X|$. However, since $F$ is not in the second $C_4$, we also have $|F|\ge 2\lambda -|X'| \ge 2 \lambda - \card{X}$, giving a contradiction. We conclude that there is at most one rank-$2$ $C_4$ in $G(\cF)$, completing the proof of the lemma.
\end{proof}

\subsubsection{Proof of part (iii)} \label{sec:k2iii}

We now wish to show that when $\lambda = o(n)$, a $2$-almost $\lambda$-Fisher family $\cF$ can have at most $\left( \frac32 + o(1) \right)n$ sets.  We will in fact prove that such a family can have size at most 
\[ \frac32 n + 3 \lambda + \frac12 \sqrt{ \lambda n } + 90. \]
Our proof is by induction on the number of pairs $F_i, F_j \in \cF$ with $F_i \subset F_j$ and $\card{F_j \setminus F_i} \ge 2$.  Recall from our previous discussion that we may assume $\lambda \ge 1$ and that $|F|>\lambda$ for every $F \in \cF$. Moreover, we may assume $m\ge 6$, since otherwise we have nothing to prove.

For the induction step, suppose $F_i \subset F_j$ with $\card{F_j \setminus F_i} \ge 2$. Note that $\card{F_i \cap F_j}=|F_i|>\lambda$, and so both $F_i$ and $F_j$ can have at most one other bad intersection.  Thus all but at most two sets in $\cF\setminus\{F_i,F_j\}$ intersect both $F_i$ and $F_j$ in precisely $\lambda$ elements. Given such a set $F$, since $F_i\subset F_j$, $\card{F \cap F_i} = \card{F \cap F_j}$ implies $F \cap (F_j \setminus F_i) = \emptyset$, and we see that there are at most $3$ sets (including $F_j$) meeting $F_j\setminus F_i$. Removing these sets, we obtain a $2$-almost $\lambda$-Fisher family of at least $m-3$ sets on the universe $[n]\setminus (F_j\setminus F_i)$, which has size at most $n-2$.  This family also has fewer nested pairs $F \subset F'$ with $\card{F' \setminus F} \ge 2$, and so by induction
\[m-3\le \frac{3}{2}(n-2)+3\lambda+\frac12 \sqrt{\lambda (n-2)}+ 90 \Rightarrow m\le \frac{3}{2}n+3\lambda + \frac12 \sqrt{\lambda n}+ 90.\]

Now for the base case we have that if $F_i\subset F_j$ for some $i \neq j$, then $|F_j\setminus F_i|=1$.  We require the following lemma, bounding the number of rank-$3$ $C_5$'s.

\begin{lemma} \label{lem:boundingC5s}
Let $\cF$ be a $2$-almost $\lambda$-Fisher family over $[n]$ as in Lemma~\ref{lemma:boundingthenumberoflowrankstructures}.  If in addition $\card{F' \setminus F} = 1$ for any $F \subset F'$ in $\cF$, then there are at most $2 \lambda + \sqrt{\lambda n} + 175$ rank-$3$ $C_5$'s in $\cF$.
\end{lemma}

Given this lemma, we can then use \eqref{eqn:withC5} to obtain the desired bound.  From Lemma \ref{lemma:boundingthenumberoflowrankstructures}, we know the number of rank-$1$ $P_1$'s is bounded by $p \le 4 \min \left\{ \lambda, \frac{n-\lambda}{3} \right\} \le 4 \lambda$, while the number of rank-$2$ $C_4$'s is at most $q \le 1$.  Lemma \ref{lem:boundingC5s} bounds the number $r$ of rank-$3$ $C_5$'s.  Substituting these bounds into \eqref{eqn:withC5} gives the required result:
\[ m \le \frac32 n + \frac{p}{2} + q + \frac{r}{2} + \frac{3}{2} \le \frac32n + 3 \lambda + \frac12 \sqrt{\lambda n} + 90. \]

\begin{proof}[Proof of Lemma \ref{lem:boundingC5s}]
We seek to bound the number of rank-$3$ $C_5$'s.  The following lemma, proven in Appendix \ref{app:appendix}, shows they have a very particular structure.

\begin{lemma} \label{lem:c5structure}
Let $\cF' \subset \cF$ be a rank-$3$ $C_5$, where $\cF$ is as in Lemma~\ref{lem:boundingC5s}.  There exists a labelling of the sets $\cF' = \{ F_1, F_2, F_3, F_4, F_5 \}$ and disjoint sets $X_0, X_1, X_2, X_3, X_4$ such that
\[ F_1 = X_0 \cup X_1 \cup X_2, \; F_2 = X_0 \cup X_3 \cup X_4, \; F_3 = X_0 \cup X_2 \cup X_3, \textrm{and } F_4 = X_0 \cup X_1 \cup X_4. \]
\end{lemma}

Observe that any four sets in a rank-$3$ $C_5$ induce a $P_3$, which has three intersections of size $\lambda$ and three of size not equal to $\lambda$.  Now apply Lemma \ref{lem:c5structure} to the $C_5$ to obtain the claimed structure.  Since the set $X_0$ is common to $F_1, F_2, F_3$ and $F_4$, it follows that each pairwise intersection between these four sets has size at least $\card{X_0}$, and thus we must have $\card{X_0} \le \lambda$.  We say a rank-$3$ $C_5$ is of Type I if $\card{X_0} = \lambda$, and of Type II if $\card{X_0} < \lambda$.  We shall show there are at most $175$ rank-$3$ $C_5$'s of Type I and at most $2 \lambda + \sqrt{\lambda n}$ of Type II, thus proving Lemma \ref{lem:boundingC5s}. \\

\noindent \underline{Type I:} $\card{X_0} = \lambda$.

We first handle the case where the common core $X_0$ has size $\lambda$, showing that there are at most $175$ rank-$3$ $C_5$'s of Type I.  Since $X_0 = F_1 \cap F_2 = F_3 \cap F_4$, it follows that these pairs are not adjacent in the cycle $G(\cF')$.  As four sets  in a $C_5$ induce three non-adjacent pairs, we may assume $F_1$ and $F_4$ are non-adjacent as well, so that the sets are $F_1, F_3, F_2, F_4$ and $F_5$ in cyclic order.

Since $F_1$ and $F_4$ are non-adjacent, we have $\card{F_1 \cap F_4} = \card{X_0} + \card{X_1} = \lambda$ as well.  Hence $X_1$ is empty, and thus $F_1 = X_0 \cup X_2$ and $F_4 = X_0 \cup X_4$.  This implies $F_1 \subset F_3 = X_0 \cup X_2 \cup X_3$ and, by our condition on $\cF$, we must have $\card{X_3} = \card{F_3 \setminus F_1} = 1$.  Finally, since $F_5$ is adjacent to $F_1$ but not $F_3$, we must have $\card{F_5 \cap F_3} = \lambda \neq \card{F_5 \cap F_1}$.  Given that $F_1 \subset F_3$ with $\card{F_3 \setminus F_1} = 1$, we must have $\card{F_5 \cap F_1} = \lambda - 1$.  Similarly, $F_4 \subset F_2$ and $\card{F_5 \cap F_2} = \lambda$, and so $\card{F_5 \cap F_4} = \lambda - 1$.

Hence, if we write $x_i = \card{X_i}$ and $\card{F_5} = \lambda + f$, the intersection matrix takes the form
\[ M(\cF') = M( \{ F_1, F_3, F_2, F_4, F_5 \} ) = \begin{pmatrix}
	x_2 & x_2 & 0 & 0 & -1 \\
	x_2 & x_2 + 1 & 1 & 0 & 0 \\
	0 & 1 & x_4 + 1 & x_4 & 0 \\
	0 & 0 & x_4 & x_4 & -1 \\
	-1 & 0 & 0 & -1 & f
\end{pmatrix}. \]
Since $\cF'$ is a rank-$3$ $C_5$, $M(\cF')$ has rank $3$.  As $\card{F_1}, \card{F_4} > \lambda$, we must have $x_2, x_4 > 0$.  The first three columns are then easily seen to be independent (consider the last three rows), and span the fourth column.  However, for $M(\cF')$ to have rank $3$, they must also span the fifth column, which is true if and only if $\frac{1}{x_2} + \frac{1}{x_4} + 1 = f$.  The only positive integer solutions to this equation are $(x_2, x_4, f) = (1,1,3)$ or $(x_2, x_4, f) = (2,2,2)$.

Now let $U = \cup_{i=0}^4 X_i$ be the support of the four sets $F_1, F_2, F_3$ and $F_4$.  Define $X_5 = F_5 \setminus U$ and $X_6 = X_0 \setminus F_5$, and let $X = X_2 \cup X_4 \cup X_5 \cup X_6$.  Observe that since $\card{F_5 \cap U} \ge \card{F_5 \cap F_3} = \lambda$, we have $\card{X_5} \le \card{F_5} - \lambda = f$.  Furthermore, since 
\[\lambda - 1 = \card{F_5 \cap F_1} = \card{F_5 \cap X_0} + \card{F_5 \cap X_2} \le \card{F_5 \cap X_0} + \card{X_2} = \card{F_5 \cap X_0} + x_2,\]
we have $\card{F_5 \cap X_0} \ge \lambda - x_2 - 1$, and hence $\card{X_6} = \card{X_0} - \card{F_5 \cap X_0} \le x_2 + 1$.  This gives $\card{X} \le 2x_2 + x_4 + f + 1$, which, given the two possibilities for $(x_2, x_4, f)$, can be bounded above by $\card{X} \le 9$.

We now claim that if $G_1$ and $G_2$ are two sets such that $\card{F_i \cap G_j} = \lambda$ for all $1 \le i \le 5$, $1 \le j \le 2$, and $G_1 \cap X = G_2 \cap X$, then $\card{G_1 \cap G_2} > \lambda$.  Observe that any set $F \in \cF \setminus \cF'$ not in the given rank-$3$ $C_5$ must intersect each of the five sets $F_i$ in exactly $\lambda$ elements.  Moreover, if there are $\ell$ other Type I rank-$3$ $C_5$'s, labelled as above, then the sets corresponding to $F_1, F_4$ and $F_5$ have pairwise intersections of size either $\lambda$ or $\lambda - 1$.  We thus obtain a family of $3 \ell$ sets whose pairwise intersections are of size at most $\lambda$, and hence by the claim they must have different intersections with the set $X$.  The number of rank-$3$ $C_5$'s is therefore bounded by $\ell + 1 \le \frac13 2^{\card{X}} + 1 \le \frac13 2^9 + 1 < 175$, as required.

Now we prove the claim.  First observe that since $\card{G_j \cap F_1} = \card{G_j \cap F_3} = \lambda$ and $X_3 = F_3 \setminus F_1$, we must have $G_j \cap X_3 = \emptyset$.  Now let $t = \card{G_1 \cap X_2}$.  Since $\lambda = \card{G_1 \cap F_1} = \card{G_1 \cap X_0} + \card{G_1 \cap X_2}$, we have $\card{G_1 \cap X_0} = \lambda - t$.  By considering $\card{G_1 \cap F_4}$, it follows that $\card{G_1 \cap X_4} = t$ as well, and hence $\card{G_1 \cap U} = \lambda + t$.  Since $G_2 \cap X_2 = G_1 \cap X_2$, we also have $\card{G_2 \cap U} = \lambda + t$.

As $G_1$ and $G_2$ share the $2t$ elements in $G_1 \cap (X_2 \cup X_4)$, we deduce
\[ \card{G_1 \cap G_2 \cap U} = \card{G_1 \cap G_2 \cap (X_0 \cup X_2 \cup X_4)} = \card{G_1 \cap (X_2 \cup X_4)} + \card{X_0} - \card{(X_0 \setminus G_1) \cup (X_0 \setminus G_2)} \ge \lambda, \]
with equality if and only if the two $t$-sets $X_0 \setminus G_1$ and $X_0 \setminus G_2$ are disjoint.  In other words, we cannot have $G_1$ and $G_2$ both missing an element in $X_0$.  We know $G_1$ and $G_2$ have the same intersection with $X_6 \subset X_0$, and hence we must have $X_6 \subset G_1 \cap G_2$.

To complete the argument, consider the intersection with $F_5$.  Let $s = \card{F_5 \cap X_2}$.  Since $\lambda - 1 = \card{F_5 \cap F_1} = \card{F_5 \cap F_4}$, it follows that $\card{F_5 \cap X_0} = \lambda - s -1$ and $\card{F_5 \cap X_4} = s$.  As $X_6 = X_0 \setminus F_5 \subset G_1$, we have $X_0 \subset F_5 \cup G_1$, and thus $\card{F_5 \cap G_1 \cap X_0} = \card{F_5 \cap X_0} + \card{G_1 \cap X_0} - \card{X_0} = \lambda - t - s - 1$.  We also have $\card{F_5 \cap G_1 \cap X_2} \le \card{F_5 \cap X_2} = s$, and $\card{F_5 \cap G_1 \cap X_4} \le \card{G_1 \cap X_4} = t$, giving
\[ \card{F_5 \cap G_1 \cap U} = \card{F_5 \cap G_1 \cap (X_0 \cup X_2 \cup X_4)} \le (\lambda - s - t - 1) + s + t = \lambda - 1. \]
Hence $F_5$ and $G_1$ must have at least one element outside $U$ in common, and thus $G_1 \cap X_5 \neq \emptyset$.  Since $G_1 \cap X = G_2 \cap X$, we have $\card{G_1 \cap G_2 \cap X_5} \ge 1$.  Thus $\card{G_1 \cap G_2} \ge \card{G_1 \cap G_2 \cap U} + \card{G_1 \cap G_2 \cap X_5} \ge \lambda + 1$, as desired. \\

\noindent \underline{Type II:}    $\card{X_0} < \lambda$.

We now turn our attention to rank-$3$ $C_5$'s of Type II, where the common core $X_0$ from Lemma \ref{lem:c5structure} has size strictly smaller than $\lambda$.  We shall show that there are at most $2 \lambda + \sqrt{\lambda n}$ such $C_5$'s.

Since $X_0 = F_1 \cap F_2 = F_3 \cap F_4$, it follows that $\card{F_1 \cap F_2}, \card{F_3 \cap F_4} \neq \lambda$, and hence these pairs are adjacent in the cycle $C_5$.  Without loss of generality we may assume that $F_{2}$ and $F_{3}$ are also adjacent so that $F_1,F_2,F_3,F_4,F_5$ form a rank-$3$ $C_5$ in this order. We thus have $|F_1\cap F_3|=|F_1\cap F_4|=|F_2\cap F_4|=\lambda$.  Since $F_1 \cap F_4 = X_0 \cup X_1$, it follows that $\card{X_1} = \lambda - \card{X_0}$.  Similarly, $\card{X_2} = \card{X_4} = \lambda - \card{X_0}$.  Thus $\card{F_1} = \card{X_0} + \card{X_1} + \card{X_2} = 2\lambda-|X_0|$,
and $|F_4|=\card{X_0} + \card{X_1} + \card{X_4} = 2\lambda-|X_0|$. Moreover, for any set $F\in \mathcal{F}\setminus \cF'$, $\card{F \cap F_1} = \card{F \cap F_2} = \lambda$, and so
\[|F|\ge |F\cap (F_1\cup F_2)|=|F\cap F_1|+|F\cap F_2|-|F\cap X_0|\ge 2\lambda-|X_0|,\]
with equality if and only if $X_0 \subset F\subset F_1 \cup F_2 = \cup_{j=0}^4 X_j$.

Suppose there are $l$ rank-$3$ $C_5$'s of Type II. For each $j \in [\ell]$, let $\{F_i^j\}_{i \in [5]}$ be the sets of the $C_5$ in cyclic order, and let $\{ X_i^j \}_{i=0}^4$ be the five corresponding sets given by Lemma \ref{lem:c5structure}.  Let $U^j = \cup_{i=0}^4 X_i^j$ be the support of the first four sets of the cycle.  By the above discussion, $|F^{j}_1|=|F^{j}_4|=2\lambda-|X^{j}_0|$ and $|F|\ge 2\lambda -|X^{j}_0|$ for every set in $\mathcal{F}\setminus \{F^{j}_i\}_{i\in [5]}$, with equality if and only if $X^{j}_0\subset F\subset U^j$. Note that for $i \in \{1,4\}$ and $j,j'\in [\ell]$ we have $2\lambda-|X^{j}_0|=|F^{j}_i|\ge 2\lambda-|X^{j'}_0|$, implying $|X^{j}_0|\le|X^{j'}_0|$, and so by symmetry $|X^{j}_0|=|X^{j'}_0|$. Moreover, it follows from this equality that $F^{j}_i\subset U^{j'}$ for $i\in \{1,4\}$. Thus all the sets in the subfamily $\mathcal{G}=\cup_{j\in [\ell]}\{F^{j}_1,F^{j}_4\}$ are supported on the universe $U =\cap_{j\in [\ell]}U^{j}$. Moreover, all the pairwise intersections of sets in $\mathcal{G}$ have size exactly $\lambda$ and so $\mathcal{G}$ is a $\lambda$-Fisher family. By Theorem \ref{thm:fisher}, we have $2\ell=|\mathcal{G}|\le |U|$. Note that for every $j\in [\ell]$,
\[ |U| \le |U^{j}| = \sum_{i=0}^4 |X^{j}_i| = \card{X^j_0} + 3 \left( \lambda - \card{X^j_0} \right) + \card{X^j_3} \le 3 \lambda + \card{X^j_3}. \]
Thus, setting $s=\min_{j\in [\ell]}\card{X^{j}_3}$, we have $\ell = \frac{1}{2} \card{ \cG } \le \frac{1}{2}\card{U} \le \frac{1}{2} (3 \lambda+s)$.  If $s\le \lambda+2 \sqrt{\lambda n}$, then we get the claimed bound on $\ell$. Hence we may assume $s>\lambda+2\sqrt{\lambda n}$.

For every $j\in [ \ell ]$ we then have $\card{F^{j}_3}=\card{X^j_0} + \card{X^j_2} + \card{X^j_3}=\lambda+\card{X^j_3} \ge \lambda+s$.  Moreover, for $j \neq j'$, we must have $\card{F^j_3 \cap F^{j'}_3} = \lambda$. Since all these sets $\{F_3^j\}_{j \in [\ell]}$ are large and have small pairwise intersections, there cannot be too many.  Indeed, by the Bonferroni inequalities, for any $\ell_0 \le \ell$,
\[
n\ge \card{\cup_{j=1}^{\ell_0}F^{j}_3} \ge \sum_{j=1}^{\ell_0} \card{F^{j}_3} - \sum_{1\le j<j'\le \ell_0} \card{F^{j}_3\cap F^{j'}_3}\ge \ell_0(\lambda+s)-\binom{\ell_0}{2}\lambda\ge \ell_0s-\frac{1}{2}\ell^2_0\lambda.
\]
Let $\ell_0 = \min \left\{ \ell, \floor{ \frac{s}{\lambda} } \right\}$, so that $n \ge \ell_0 s - \frac12 \ell_0^2 \lambda \ge \frac12 \ell_0 s$.  We cannot have $\ell_0 = \floor{ \frac{s}{\lambda}}$, for then, given $s > \lambda + 2 \sqrt{ \lambda n }$, we reach a contradiction:
\[n\ge \frac{1}{2} \ell_0 s \ge \frac{1}{2} \left(\frac{s}{\lambda}-1\right) s >\frac{1}{2} \left(2 \sqrt{\frac{n}{\lambda}} \right) \left( 2 \sqrt{\lambda n} \right) = 2n. \]
Hence we must have $\ell_0 = \ell$, which gives
\[ \ell = \ell_0 \le \frac{2n}{s}<\frac{2n}{\lambda+2\sqrt{\lambda n}}< \sqrt{\frac{n}{\lambda}}<2\lambda+\sqrt{\lambda n}.\]

Hence there are at most $2\lambda+\sqrt{\lambda n}$ rank-$3$ $C_5$'s of Type II, and thus at most $2\lambda + \sqrt{\lambda n} + 175$ rank-$3$ $C_5$'s in total, completing the proof of the lemma.
\end{proof}

\section{When $k$ is large} \label{sec:largek}

We will show how to use our results to improve the upper bound for $f(n,k,\lambda)$ for larger values of $k$.  Let $\cF$ be a $k$-almost $\lambda$-Fisher family over $[n]$.  Recall that Proposition \ref{prop:trivial} gives the upper bound $\card{\cF} \le (k+1)n+1$ by using lower bounds on the independence number of graphs with bounded degree to find a relatively large $\lambda$-Fisher subfamily $\cF' \subset \cF$, and then using Theorem \ref{thm:fisher} to bound the size of $\cF'$.  We shall instead use the following partitioning result of Lov\'asz \cite{l66}.

\begin{thm}[Lov\'asz \cite{l66}, 1966] \label{thm:partition}
Let $G$ be a graph of maximum degree $\Delta(G) = \Delta$.  Then, for any $t \ge 1$ and integers $\Delta_i$ such that $\sum_{i=1}^t (\Delta_i + 1) \ge \Delta + 1$, there is a partition of the vertices $V(G) = \cup_{i=1}^t V_i$ such that the maximum degrees of the induced subgraphs $G[V_i]$ satisfy $\Delta(G[V_i]) \le \Delta_i$.
\end{thm}

This theorem, coupled with our bounds on $f(n,2,\lambda)$, allows us to prove Corollary \ref{cor:largek}.

\setcounter{section}{1}
\setcounter{thm}{6}
\begin{cor}
For $k \ge 1$, we have $f(n,k,\lambda) \le (2n - 2) \ceil{\frac{k+1}{3}}$.  Moreover, if $\lambda = o(n)$, then $f(n,k,\lambda) \le \left(\frac32 + o(1)\right) n \ceil{\frac{k+1}{3}}$.
\end{cor}
\setcounter{section}{5}
\setcounter{thm}{1}

\begin{proof}
Let $\cF$ be a $k$-almost $\lambda$-Fisher family over $[n]$.  The auxiliary graph $G = G(\cF)$ has maximum degree $k$, and hence by Theorem \ref{thm:partition}, the vertices of $G$ can be partitioned into $\ceil{\frac{k+1}{3}}$ induced subgraphs of maximum degree at most $2$.  This corresponds to partitioning $\cF$ into $\ceil{\frac{k+1}{3}}$ $2$-almost $\lambda$-Fisher subfamilies.  By part (i) of Theorem \ref{thm:kequalstwo}, each such family can have size at most $2n-2$, and hence we have $\card{\cF} \le (2n-2) \ceil{\frac{k+1}{3}}$.

Moreover, if $\lambda = o(n)$, then by part (iii) of Theorem \ref{thm:kequalstwo}, each of the subfamilies can have size at most $\left( \frac32 + o(1) \right) n$, giving rise to the improved bound $\card{\cF} \le \left( \frac32 + o(1) \right) n \ceil{\frac{k+1}{3}}$.

As $\cF$ was an arbitrary $k$-almost $\lambda$-Fisher family, the desired bounds on $f(n,k,\lambda)$ follow.
\end{proof}

\section{Concluding remarks} \label{sec:conc}

In this paper we bound the size of $k$-almost $\lambda$-Fisher families when $k$ is small, making progress on a problem introduced by Vu \cite{v99}.  Vu showed that the largest $1$-almost $\lambda$-Fisher families are given by the Hadamard construction, and we show that the same construction remains optimal for $k = 2$.  One might ask whether, as in the case $k = 1$, the Hadamard construction is the unique $2$-almost $\lambda$-Fisher family of size $n$.

Our proof shows that any $2$-almost $\lambda$-Fisher family of $2n-2$ sets must consist mostly of rank-$1$ $P_1$'s, with perhaps a few larger low-rank structures.  However, we believe that the presence of these larger structures would place too many restrictions upon the other sets in the family, and hence the bound of $2n-2$ can only be obtained by a family of rank-$1$ $P_1$'s; that is, by the Hadamard construction.

Let us now consider $3$-almost $\lambda$-Fisher families.  By Theorem \ref{thm:partition}, any such family may be partitioned into a $2$-almost $\lambda$-Fisher family and a $\lambda$-Fisher family.  Thus $f(n,3,\lambda) \le f(n,2,\lambda) + f(n,0,\lambda) \le 3n-2$.  However, we have not found a $3$-almost $\lambda$-Fisher family larger than the Hadamard construction, and so it may well be that the Hadamard construction is still optimal for $k=3$.  If this is the case, though, we know that it is not the unique optimal family.  To see this, let $n = 4m + t$, $t \le 4m-1$, and let $\lambda= m$.  Start by taking $\cF_0 = \{F_{i,j} : i \in [4m-1], j \in [2]\}$ to be the Hadamard construction on $[4m]$, where $\{ F_{i,1}, F_{i,2} \}_{i \in [4m-1]}$ are the disjoint pairs.  Now, for each of the $t$ elements in $[n] \setminus [4m]$, add the sets $\cF_1 = \{ F_{i,j} \cup \{4m + i\} : i \in [t], j \in [2] \}$.  It is then easy to verify that $\cF = \cF_0 \cup \cF_1$ is a $3$-almost $m$-Fisher family of size $2n-2$, thus matching the Hadamard construction.

The main open problem, though, is to determine the behaviour of $f(n,k,\lambda)$ for large $k$.  From the best known constructions (see Section 2 and  \cite{v99}), it is natural to conjecture that
\[ \lim_{k \rightarrow \infty} \lim_{n \rightarrow \infty} \max_{0 \le \lambda \le n} \frac{f(n,k,\lambda)}{kn} = \frac14. \]
We have shown that this holds for $\lambda = 0$, and give some evidence that this case should represent the typical behaviour of $f(n,k,\lambda)$.  It would be very interesting to resolve this problem for all $k$.

\paragraph{Acknowledgements} We would like to thank the anonymous referees for their careful reading of our manuscript and their helpful suggestions for improving the presentation of this paper.

\appendix

\section{$\{0,1\}$-vectors and low-rank structures} \label{app:appendix}

Here we prove the techincal lemmas needed in Section \ref{sec:kequalstwo}. 

\renewcommand{\thesection}{\arabic{section}}
\setcounter{section}{4}
\setcounter{thm}{0}
\begin{lemma}
Let $v_i\in \{0,1\}^n$, $i\in [5]$, be five distinct non-zero vectors. Suppose that there exist $\lambda\in \RR\setminus\{0\}$ and $v\in \RR^n$ such that $v\cdot v_i=\lambda$ for $i\in [5]$. Then:
\begin{enumerate}[(a)]
\item The vectors $\{ v_i \}_{i \in [3]}$ are linearly independent.
\item The vectors $v_1-v_2$ and $v_1-v_3$ are linearly independent.
\item If the vectors $\{v_i\}_{i\in [4]}$ are linearly dependent, then $v_1+v_2=v_3+v_4$ holds for some relabelling of these four vectors.
\item Four of the vectors $\{v_i\}_{i\in [5]}$ are linearly independent.
\end{enumerate}
\end{lemma}

\begin{proof}
First we show (a). Assume for the sake of contradiction that $v_1$, $v_2$ and $v_3$ are linearly dependent. Since any two of the vectors are distinct non-zero $\{0,1\}$-vectors, and hence linearly independent, there must exist non-zero real numbers $\alpha_1$ and $\alpha_2$ such that $\alpha_1v_1+\alpha_2v_2=v_3$.  By taking the inner product of this relation with $v$, we conclude that $\alpha_1+\alpha_2=1$.

Since $v_1\neq v_2$, we may assume without loss of generality that there is one coordinate which is $1$ in $v_1$ and $0$ in $v_2$. The value of that coordinate in $\alpha_1v_1+\alpha_2v_2$ is thus $\alpha_1$ and, since the linear combination is equal to the $\{0,1\}$-vector $v_3$, we conclude that $\alpha_1=0$ or $\alpha_1=1$. If $\alpha_1=0$, then $\alpha_2=1$, and so $v_2=v_3$, contradicting the fact that $v_2\neq v_3$. Similarly, if $\alpha_1=1$, then $v_1=v_3$. This settles (a).

(b) follows easily from (a) since a non-trivial linear relation between the vectors $v_1-v_2$ and $v_1-v_3$ would contradict the independence of $v_1, v_2$ and $v_3$.

Next we prove (c). Suppose the vectors $\{v_i\}_{i\in [4]}$ are linearly dependent. By (a), any three of these four vectors are linearly independent, and so there must be a relation involving all four vectors. Writing this relation with positive coefficients, we see there must be positive $\alpha_1$, $\alpha_2$ and $\alpha_3$ such that, after relabelling, one of the following relations holds:
\[\text{(i) }\alpha_1v_{1}+\alpha_2v_{2}+\alpha_3v_{3}=v_{4}\;\;\text{ or }\;\;\text{(ii) }\alpha_1v_{1}+\alpha_2v_{2}=\alpha_3v_{3}+v_{4}.\]
We first show that case (i) is impossible. Note that by taking inner products with $v$, we may conclude that $\alpha_1+\alpha_2+\alpha_3=1$. Since $v_1 \neq v_2$, there is, without loss of generality, some coordinate which is $1$ in $v_1$ and $0$ in $v_2$.  In this coordinate, the expression on the left-hand side must then be equal to either $\alpha_1$ or $\alpha_1 + \alpha_3$; in either case, we have $0 < \alpha_1 < \alpha_1 + \alpha_3 < 1$.  Since $v_4$ is a $\{0,1\}$-vector, the right-hand side is either $0$ or $1$, giving a contradiction.

We now show that in case (ii), we must have $\alpha_1=\alpha_2=\alpha_3=1$. As before, taking the inner product with $v$ gives $\alpha_1+\alpha_2=\alpha_3+1$. Moreover, since $v_{1}\neq v_{2}$, we may assume without loss of generality that there is one coordinate which is $1$ in $v_{1}$ and $0$ in $v_{2}$. The value of this coordinate in $\alpha_1v_{1}+\alpha_2v_{2}$ is thus $\alpha_1$, and in $\alpha_3 v_{3} + v_{4}$ is one of $\{0,\alpha_3,1,\alpha_3 + 1\}$. Since the coefficients $\alpha_i$ are positive and $\alpha_1+\alpha_2=\alpha_3+1$ it follows that $\{\alpha_1,\alpha_2\}=\{\alpha_3,1\}$. Suppose $\alpha_1=\alpha_3$ and $\alpha_2=1$ (the other case follows analogously). Since $v_{1}\neq v_{3}$ there is a coordinate where they differ. Repeating the above argument then shows that either $\alpha_1=1$ or $\alpha_3=\alpha_2$, implying $\alpha_1=\alpha_2=\alpha_3=1$, thus proving (c).

Finally, we prove (d). Suppose for the sake of contradiction that all sets of  four of the vectors $v_i$ are linearly dependent. Then, by part (c) applied to the vectors $\{v_1,v_2,v_3,v_4\}$ we may assume, up to some permutation of the indices, that $v_1+v_2=v_3+v_4$.  The vectors $\{v_1, v_2, v_3, v_5\}$ must also contain a similar relation.  If it is $v_1 + v_2 = v_3 + v_5$, then we have $v_4 = v_1 + v_2 - v_3 = v_5$, contradicting the fact that these vectors are distinct.  Hence, by symmetry, we may assume we have $v_1 + v_3 = v_2 + v_5$.  Adding the two relations gives $2v_1=v_4+v_5$, contradicting the linear independence of the vectors $\{v_1,v_4,v_5\}$ guaranteed by (a).
\end{proof}

\begin{lemma}
The ranks of the components can be bounded as follows:
\begin{enumerate}[(a)]
\item If $\cF'$ is the $s$-vertex path $P_{s-1}$, $\rank(M(\cF')) \ge s-1$.
\item If $\cF'$ is the $s$-vertex cycle $C_s$, $\rank(M(\cF')) \ge s-2$.
If $\cF'$ is the triangle $C_3$, and there is some set $F$ whose intersections with every set in $\cF'$ all have size $\lambda$, then $\rank(M(\cF')) \ge 2$.\end{enumerate}
\end{lemma}

\begin{proof}
We begin with $(a)$.  If we order the sets according to the path, the matrix $M = M(\cF')$ takes the tridiagonal form
\[ \begin{pmatrix}
	\ast & \ast & 0 & 0 & \hdots & 0 \\
	\ast & \ast & \ast & 0 & \hdots & 0 \\
	0 & \ast & \ast & \ast & \ddots & \vdots \\
	0 & 0 & \ast & \ast & \ddots & 0 \\
	\vdots & \vdots & \ddots & \ddots & \ddots & \ast \\
	0 & 0 & \hdots & 0 & \ast & \ast
	\end{pmatrix}, \]
where $\ast$ denotes a non-zero entry.  Deleting the top row and right column leaves a non-singular upper-diagonal $(s-1) \times (s-1)$ matrix, and thus $\rank (M) \ge s-1$, as claimed.

Next we prove $(b)$.  If we order the sets cyclically, the matrix $M = M(\cF')$ takes the almost-tridiagonal form
\[ \begin{pmatrix}
	\ast & \ast & 0 & \hdots & 0 & \ast \\
	\ast & \ast & \ast & 0 & \hdots & 0 \\
	0 & \ast & \ast & \ast & \ddots & \vdots \\
	\vdots & 0 & \ast & \ast & \ddots & 0 \\
	0 & \vdots & \ddots & \ddots & \ddots & \ast \\
	\ast & 0 & \hdots & 0 & \ast & \ast
	\end{pmatrix} . \]
Deleting the top two rows and the first and last columns we obtain a non-singular upper-diagonal $(s-2) \times (s-2)$ matrix, and thus $\rank (M) \ge s-2$.

Finally, we show $(c)$.  Note that the columns of $A = A(\cF')$ are the three characteristic vectors for the sets in $\cF'$.  Since $\card{F \cap F_i} = \lambda$ for each set $F_i \in \cF'$, the characteristic vectors satisfy $v \cdot v_i = \lambda$.  We may therefore apply Lemma \ref{lem:(0,1)-vectors} (a), implying these vectors are linearly independent, and so $\rank (A) = 3$.  Since $M = A^T A - \lambda J_3$, it follows from Lemma \ref{lemma:vurank} that $\rank (M) \ge 2$.
\end{proof}

\begin{lemma} \label{lem:lowrankspan}
Let $\cF$ be a $2$-almost $\lambda$-Fisher family of size $\card{\cF} \ge 5$ with $\card{F} > \lambda$ for all $F \in \cF$.  If a component $\cF' \subset \cF$ is either a rank-$1$ $P_1$ or a rank-$2$ $C_4$, then the columns of $M(\cF')$ do not span the all-$1$ vector $(1, \hdots, 1)^T$.
\end{lemma}

\begin{proof}
First suppose $\cF'$ is a rank-$1$ $P-1$.  Let $\cF' = \{F_1, F_2\}$.  $M = M(\cF')$ is then
\[ \begin{pmatrix}
\card{F_1} - \lambda & \card{F_1 \cap F_2} - \lambda \\
\card{F_1 \cap F_2} - \lambda & \card{F_2} - \lambda
\end{pmatrix} . \]
If $\rank (M) = 1$ and the columns span $(1,1)^T$, then we must have $\card{F_1} - \lambda = \card{F_1 \cap F_2} - \lambda = \card{F_2} - \lambda$, which implies $F_1 = F_2$, a contradiction.

Now let $\cF'$ be a rank-$2$ $C_4$.  Suppose, in cyclic order, we have $\cF' = \{F_1, F_2, F_3, F_4\}$.  Then $M = M(\cF')$ takes the form
\[ \begin{pmatrix}
	\card{F_1} - \lambda & \card{F_1 \cap F_2} - \lambda & 0 & \card{F_1 \cap F_4} - \lambda \\
	\card{F_1 \cap F_2} - \lambda & \card{F_2} - \lambda & \card{F_2 \cap F_3} - \lambda & 0 \\
	0 & \card{F_2 \cap F_3} - \lambda & \card{F_3} - \lambda & \card{F_3 \cap F_4} - \lambda \\
	\card{F_1 \cap F_4} - \lambda & 0 & \card{F_3 \cap F_4} - \lambda & \card{F_4} - \lambda
\end{pmatrix}. \]

Suppose $\rank (M) = 2$ and the columns span $(1,1,1,1)^T$.  Any two columns are clearly independent, and so it follows that any two should span $(1,1,1,1)^T$.  Hence it suffices to show this is not the case for the first two columns.  Suppose for contradiction we had $\alpha_1$ and $\alpha_2$ such that
\begin{equation} \label{eqn:twospan}
\alpha_1 \begin{pmatrix}]
	\card{F_1} - \lambda \\
	\card{F_1 \cap F_2} - \lambda \\
	0 \\
	\card{F_1 \cap F_4} - \lambda
\end{pmatrix}
+ \alpha_2 \begin{pmatrix}
	\card{F_1 \cap F_2} - \lambda \\
	\card{F_2} - \lambda \\
	\card{F_2 \cap F_3} - \lambda \\
	0
\end{pmatrix}
= \begin{pmatrix}
	1 \\
	1 \\
	1 \\
	1
\end{pmatrix}.
\end{equation}

Since $\card{\cF} \ge 5$, we may find a set $F\in \cF \setminus \cF'$ with $|F\cap F_i|=\lambda$ for every $i\in [4]$. By Lemma \ref{lem:(0,1)-vectors} (a), any three columns of $A = A(\cF')$ are linearly independent, and so $\rank (A) \ge 3$.  By Lemma \ref{lemma:vurank}, $\rank (M) \ge \rank (A) - 1$, and so we must have $\rank (A) = 3$.  Thus the four columns of $A$ are linearly dependent, and by Lemma \ref{lem:(0,1)-vectors} (c), it follows that we have the relation $v_{\pi(1)} + v_{\pi(2)} = v_{\pi(3)} + v_{\pi(4)}$ for some permutation $\pi\in S_4$.

Suppose first we had $v_1 + v_3 = v_2 + v_4$.  By considering the coordinates where this sum is equal to $2$, it follows that $F_1 \cap F_3 = F_2 \cap F_4$.  Since $\card{F_1 \cap F_3} = \lambda$, all pairwise intersections have size at least $\lambda$.  From the third and fourth coordinates of \eqref{eqn:twospan}, we have $\alpha_1 \left( \card{F_1 \cap F_4} - \lambda \right) = \alpha_2 \left( \card{F_2 \cap F_3} - \lambda \right) = 1$, and so we must have $\alpha_1, \alpha_2 > 0$.  However, the second coordinate then gives a contradiction:
\[ 1 = \alpha_1 \left( \card{F_1 \cap F_2} - \lambda \right) + \alpha_2 \left( \card{F_2} - \lambda \right) > \alpha_2 \left( \card{F_2} - \lambda \right) \ge \alpha_2 \left( \card{F_2 \cap F_3} - \lambda \right) = 1. \]

By symmetry, therefore, we may assume $v_1 + v_4 = v_2 + v_3$, and so $F_1 \cap F_4 = F_2 \cap F_3$, and thus $\card{F_1 \cap F_4} = \card{F_2 \cap F_3} \le \card{F_1 \cap F_2}$.  Thus, from the third and fourth coordinates of \eqref{eqn:twospan}, we have $\alpha_1 = \alpha_2 = \alpha$.  Equating the first and third coordinates, we have
\[ \alpha \left( \card{F_1} - \lambda \right) + \alpha \left( \card{F_1 \cap F_2} - \lambda \right) = 1 = \alpha \left( \card{F_2 \cap F_3} - \lambda \right), \]
and so $\card{F_1} - \lambda = \card{F_2 \cap F_3} - \card{F_1 \cap F_2} \le 0$. This implies $\card{F_1} \le \lambda$, giving the desired contradiction.
\end{proof}

\setcounter{thm}{5}
\begin{lemma}
Let $\cF' \subset \cF$ be a rank-$3$ $C_5$, where $\cF$ is as in Lemma~\ref{lem:boundingC5s}.  There exists a labelling of the sets $\cF' = \{ F_1, F_2, F_3, F_4, F_5 \}$ and disjoint sets $X_0, X_1, X_2, X_3, X_4$ such that
\[ F_1 = X_0 \cup X_1 \cup X_2, \; F_2 = X_0 \cup X_3 \cup X_4, \; F_3 = X_0 \cup X_2 \cup X_3, \textrm{and } F_4 = X_0 \cup X_1 \cup X_4. \]
\end{lemma}

\begin{proof}
We are assuming the intersection matrix $M = M(\cF')$ has rank $3$.  Since $M = A^T A - \lambda J_m$, where $A = A(\cF')$, by Lemma \ref{lemma:vurank} we have $\rank(M) \ge \rank(A) - 1$, and so $\rank(A) \le 4$.   The columns of $A$ are the characteristic vectors of the sets in $\cF'$, and hence we may deduce that these vectors are linearly dependent.

Suppose first that there are four of the sets whose vectors are linearly dependent.  By Lemma \ref{lem:(0,1)-vectors} (c), we can label the sets such that the characteristic vectors satisfy $v_1 + v_2 = v_3 + v_4$.  By considering the coordinates where the sum is positive, and is equal to $2$, we easily deduce that $F_1 \cup F_2 = F_3 \cup F_4$ and $F_1 \cap F_2 = F_3 \cap F_4$.  Let $X_0 = F_1 \cap F_2 = F_3 \cap F_4$, and let $F_i' = F_i \setminus X_0$ for $i \in [4]$.

Let $X_1 = F_1' \cap F_4'$, $X_2 = F_1' \cap F_3'$, $X_3 = F_2' \cap F_3'$ and $X_4 = F_2' \cap F_4'$.  Observe that since $F_1' \cap F_2' = \emptyset = F_3' \cap F_4'$, the sets $\{ X_i \}_{i = 0}^4$ are disjoint.  Since $F_1' \subset F_3' \cup F_4'$, we have $F_1' = F_1' \cap (F_3' \cup F_4') = X_1 \cup X_2$.  Adding back the common core $X_0$ gives $F_1 = X_0 \cup X_1 \cup X_2$, as required.  The remaining equalities follow similarly.  To prove the lemma, we shall show that we must always have a relation between the vectors of four of the sets; that is, there cannot be a minimal relation involving all five sets.

Suppose for contradiction we had such a relation.  Writing the relation with positive coefficients, it must either take the form, for some labelling of the sets,
\begin{equation} \label{eqn:case1}
\alpha_1 v_1 + \alpha_2 v_2 + \alpha_3 v_3 + 
\alpha_4 v_4 = v_5
\end{equation}
or
\begin{equation} \label{eqn:case2}
\alpha_1 v_1 + \alpha_2 v_2 + \alpha_3 v_3 = 
\alpha_4 v_4 + \alpha_5 v_5.
\end{equation}

First consider the relation in \eqref{eqn:case1}.  If $F \in \cF \setminus \cF'$ is another set in the family, then $\card{F \cap F_i} = \lambda$ for $i \in [5]$.  Taking dot products of \eqref{eqn:case1} with the vector $v_F$, we must have $\sum_{i \in [4]} \alpha_i \lambda = \lambda$, and so $\sum_{i \in [4]} \alpha_i = 1$.  Now, since $F_1 \neq F_2$, we may without loss of generality assume there is some element $j \in F_1 \setminus F_2$.  The $j$th coordinate on the left-hand side is between $\alpha_1$ and $1 - \alpha_2$, where $0 < \alpha_1 < 1 - \alpha_2 < 1$.  However, since $v_5$ is a $\{0,1\}$-vector, the same coordinate on the right-hand side is either $0$ or $1$, and hence the two cannot be equal, giving the desired contradiction.

Now consider \eqref{eqn:case2}.  By the same argument as before, we have $\alpha_1 + \alpha_2 + \alpha_3 = \alpha_4 + \alpha_5$.  The positivity of the $\alpha_i$'s implies $F_1 \cup F_2 \cup F_3 = F_4 \cup F_5$ and $F_1 \cap F_2 \cap F_3 = F_4 \cap F_5$.  Let $X_0$ denote this common core, and let $X_i = F_i \setminus X_0$ for $i \in [5]$.  Note that we have $\card{F_i \cap F_j} \ge \card{X_0}$ for $i,j \in [5]$.  Since $\cF'$ forms a five-cycle, we must have some pairwise intersections equal to $\lambda$, and hence we deduce $\card{X_0} \le \lambda$.

First we rule out the case $\card{X_0} = \lambda$.  In this case, we have $\card{F_4 \cap F_5} = \card{X_0} = \lambda$, and hence $F_4$ and $F_5$ cannot be adjacent in the five-cycle $G(\cF')$. Without loss of generality, we may assume the sets in cyclic order are $F_5, F_2, F_3, F_4$ and $F_1$.  Since $F_2$ and $F_3$ are adjacent, we must have $\card{F_2 \cap F_3} \neq \lambda$.  They already share the core $X_0$ of size $\lambda$, and thus we must have $X_2 \cap X_3 \neq \emptyset$.  As $F_1 \cup F_2 \cup F_3 = F_4 \cup F_5$, it follows that $X_2 \cap X_3 \subset X_4 \cup X_5$.  However, if $X_2 \cap X_3 \cap X_4 \neq \emptyset$, then we have $\card{F_2 \cap F_4} = \card{X_0} + \card{X_2 \cap X_4} > \card{X_0} = \lambda$, contradicting the fact that $F_2$ and $F_4$ are not adjacent in $G(\cF')$.  Similarly, if $X_2 \cap X_3 \cap X_5 \neq \emptyset$, we have $\card{F_3 \cap F_5} > \lambda$.  Hence we cannot have $\card{X_0} = \lambda$.

Finally, suppose $\card{X_0} < \lambda$.  Hence $\card{F_4 \cap F_5} < \lambda$, and so $F_4$ and $F_5$ must be adjacent in $G(\cF')$.  We may therefore assume the five-cycle consists of $F_1, F_2, F_3, F_4$ and $F_5$ in cyclic order.  Since $F_1$ and $F_3$ are not adjacent, we must have $\lambda = \card{F_1 \cap F_3} = \card{X_0} + \card{X_1 \cap X_3}$.  Since $\card{X_0} < \lambda$, it follows that $X_1 \cap X_3 \neq \emptyset$.  Note that since $X_0 = F_1 \cap F_2 \cap F_3$, we must have $X_1 \cap X_3$ disjoint from $X_2$.  Moreover, we have $X_1 \cap X_3 \subset X_4 \cup X_5$, and without loss of generality we may assume $X_1 \cap X_3 \cap X_4 \neq \emptyset$.

Consider any element $j \in X_1 \cap X_3 \cap X_4$.  By considering the $j$th coordinate in \eqref{eqn:case2}, we have $\alpha_1 + \alpha_3 = \alpha_4$.  Since $\alpha_1 + \alpha_2 + \alpha_3 = \alpha_4 + \alpha_5$, it follows that $\alpha_2 = \alpha_5$.  If we also have $X_1 \cap X_3 \cap X_5 \neq \emptyset$, then we would similarly have $\alpha_2 = \alpha_4$, and so there is some $\alpha = \alpha_1 + \alpha_3 = \alpha_2 = \alpha_4 = \alpha_5$.  Now $F_1 \neq F_3$, and so without loss of generality we may assume there is some $j' \in F_1 \setminus F_3$.  The corresponding coordinate in the left-hand side of \eqref{eqn:case2} must be equal to either $\alpha_1$ or $\alpha_1 + \alpha_2 = \alpha_1 + \alpha$.  On the right-hand side, the possible values are $0, \alpha$ or $2 \alpha$.  Since $0 < \alpha_1 < \alpha_1 + \alpha_3 = \alpha$, we cannot have equality, giving rise to a contradiction.

Hence we must have $X_1 \cap X_3 \subset X_4$.  Note that $F_1$ and $F_4$ are not adjacent, and hence we must have $\card{F_1 \cap F_4} = \lambda$.  We already have $\card{F_1 \cap F_3 \cap F_4} = \card{X_0} + \card{X_1 \cap X_3 \cap X_4} = \card{X_0} + \card{X_1 \cap X_3} = \card{F_1 \cap F_3} = \lambda$, and thus we must have $X_1 \cap X_4 \subset X_3$.  Now $F_2$ and $F_4$ are also not adjacent, and so we must have $\card{F_2 \cap F_4} = \lambda$, and thus $\card{X_2 \cap X_4} = \lambda - \card{X_0}$.  Consider any element $j \in X_2 \cap X_4$.  On the right-hand side of \eqref{eqn:case2}, the $j$th coordinate is equal to $\alpha_4$.  On the left-hand side it is equal to at least $\alpha_2 = \alpha_5$.  Note that these cannot be equal, as we have already seen $\alpha_2 = \alpha_4 = \alpha_5$ leads to a contradiction.  Hence we must also have $j \in X_1 \cup X_3$.  We cannot have $j \in X_1$, as $X_1 \cap X_4 \subset X_3$, and $X_1 \cap X_2 \cap X_3 = \emptyset$.  Thus $j \in X_3$, and hence we have shown $X_2 \cap X_4 \subset X_3$.

Now consider any element $j \in X_2 \cap X_5$, which must be non-empty as $\card{F_2 \cap F_5} = \lambda$.  On the right-hand side of \eqref{eqn:case2}, the $j$th coordinate has value $\alpha_5$, while on the left-hand side it has value at least $\alpha_2 = \alpha_5$.  Hence it equals $\alpha_2$, and we cannot have $j \in X_1 \cup X_3$.  Thus it follows that $X_2 \cap X_3 \cap X_5 = \emptyset$, and hence $X_2 \cap X_3 \subset X_4$.  In light of our previous observation, however, this implies $X_2 \cap X_3 = X_2 \cap X_4$, and so $F_2 \cap F_3 = F_2 \cap F_4$.  Thus $\lambda = \card{F_2 \cap F_4} = \card{F_2 \cap F_3}$, which contradicts the fact that $F_2$ and $F_3$ are adjacent in $G(\cF')$.  This completes the proof.
\end{proof}
\end{document}